\documentclass[a4paper,11pt]{amsart}
\usepackage{amssymb}
\usepackage{amsmath}
\usepackage{cancel}
\usepackage{mathrsfs}
\usepackage{color}
\usepackage{caption}
\usepackage{subcaption}
\usepackage{hyperref}
\pdfoutput=1
\usepackage[a4paper,margin=3cm]{geometry}
\numberwithin{equation}{section}
\DeclareSymbolFont{script}{U}{eus}{m}{n}
\DeclareSymbolFontAlphabet{\amathscr}{script}
\DeclareMathSymbol{\Wedge}{0}{script}{"5E}
\DeclareMathAlphabet{\mathrmsl}{OT1}{cmr}{m}{sl}

\newenvironment{dedication}
  {%\clearpage           % we want a new page
   \thispagestyle{empty}% no header and footer
   \vspace*{0.2cm}% some space at the top 
   \itshape             % the text is in italics
   \raggedleft          % flush to the right margin
  }
  {\par % end the paragraph
   \vspace{1cm} % space at bottom is three times that at the top
  }

\newtheorem{lemma}{Lemma}[section]
\newtheorem{prop}{Proposition}[section]
\newtheorem{thm}{Theorem}[section]
\newtheorem{cor}{Corollary}[section]
\newtheorem{assumption}{Assumption}[section]
\theoremstyle{definition}
\newtheorem{defn}{Definition}[section]
\theoremstyle{remark}
\newtheorem{rem}{Remark}[section]

\newtheorem{ex}{Example}[section]

\theoremstyle{Conjecture}
\newtheorem{conj}{Conjecture}[section]
\theoremstyle{Problem}
\newtheorem{problem}{Problem}[section]

\newcommand{\Ds}{{\amathscr D}}
\newcommand{\R}{{\mathbb R}}
\newcommand{\C}{{\mathbb C}}

\newcommand{\N}{{\mathbb N}}
\newcommand{\Q}{{\mathbb Q}}
\newcommand{\T}{{\mathbb T}}

\newcommand{\Sph}{{\mathbb S}}

\newcommand{\cL}{{\mathcal L}}
\newcommand{\cO}{{\mathcal O}}
\newcommand{\cR}{{\mathcal R}}

\newcommand{\Scal}{{\rm Scal}}

\newcommand{\vol}{\mathit{vol}}
\newcommand{\Vol}{\mathit{Vol}}

\newcommand{\grad}{\mathrm{grad}}

\newcommand{\kt}{\hat{\mathfrak t}}
\newcommand{\tor}{{\mathfrak t}}

\newcommand{\Hess}{\mathop{\mathrm{Hess}}}

\newcommand{\Pol}{\mathrm P}
\newcommand{\mL}{{\mathcal{ L}}}
\newcommand{\ra}{\rightarrow}

\newcommand{\Aff}{{\rm Aff}}
\newcommand{\Fut}{{\rm Fut}}
\newcommand{\Aut}{\mathrm{Aut}}

\newcommand{\tr}{{\mathrmsl tr}}

\newcommand{\Ric}{{\mathrm{Ric}}}

\begin{document}

\begin{dedication}
Dedicated to Paul Gauduchon on the occasion of his 80th birthday,\\ with respect and gratitude.  
\end{dedication}

\title[]{From K\"ahler Ricci solitons to Calabi-Yau K\"ahler cones}

\author[V. Apostolov]{Vestislav Apostolov} \address{Vestislav Apostolov \\ D{\'e}partement de Math{\'e}matiques\\ UQAM\\ C.P. 8888 \\ Succursale Centre-ville \\ Montr{\'e}al (Qu{\'e}bec) \\ H3C 3P8 \\ Canada \\ and \\ Institute of Mathematics and Informatics, Bulgarian Academy of Sciences}
\email{apostolov.vestislav@uqam.ca}

\author[A. Lahdili]{Abdellah Lahdili} \address{Abdellah Lahdili\\ D{\'e}partement de Math{\'e}matiques\\ UQAM\\ C.P. 8888 \\ Succursale Centre-ville \\ Montr{\'e}al (Qu{\'e}bec) \\ H3C 3P8 \\ Canada}
\email{lahdili.abdellah@gmail.com}

\author[E. Legendre]{Eveline Legendre}\address{Eveline Legendre\\ Universit\'e Claude Bernard Lyon 1\\
Institut Camille Jordan \'equipe AGL\\ 21 av. Claude Bernard\\
69100 Villeurbanne\\ France}
\email{eveline.legendre@math.univ-lyon1.fr}

\thanks{V.A. was supported in part by  an NSERC Discovery Grant, an FRQNT Team Grant and a Connect Talent Grant of the Region de Pays de la Loire. He is grateful to the Institut Camille Jordan Lyon 1 for hospitality during the preparation of this work. A.L. was supported by the Villum Young Investigator grant 0019098 and UQAM. E.L. was supported in part by France ANR project BRIDGES No ANR-21-CE40-0017, she is grateful to the Simons foundation, CRM and FQRNT for offering her a visiting position during the preparation of this work. The authors thank the referees for their useful remarks.}

\begin{abstract} We show that if $X$ is a smooth Fano manifold which caries a K\"ahler Ricci soliton, then the canonical cone of the product of $X$ with a complex projective space of sufficiently large dimension is a Calabi--Yau cone, i.e. admits a Ricci-flal K\"ahler cone metric. This can be seen as an asymptotic version of a conjecture by Mabuchi and Nikagawa.  This result is obtained by the openness of the set of weight functions $v$ over the momentum polytope of a given smooth Fano manifold, for which a $v$-soliton exists. We discuss other ramifications of this approach, including a Licherowicz type obstruction  to the existence of a K\"ahler Ricci soliton and a Fujita type volume bound for the existence of a  $v$-soliton.
\end{abstract}

\date{\today}
\maketitle
%\tableofcontents

\section{Introduction} A number of different notions of \emph{special} K\"ahler metrics have emerged in the last 20 years or so, in connection with Calabi's seminal program \cite{Calabi55} of finding a canonical representative of a given de Rham class of K\"ahler metrics on a smooth compact K\"ahler manifold $X$. Perhaps the most studied notion of all, introduced by Calabi himself, is that of \emph{constant scalar curvature} K\"ahler (cscK) metrics.

In the case of a smooth Fano variety $X$ endowed with its anti-canonical de Rham K\"ahler class $2\pi c_1(X)$, a cscK metric in $2\pi c_1(X)$ is necessarily a K\"ahler--Einstein metric with scalar curvature equal to $2n$. The existence problem for such K\"ahler metrics is understood in terms of the Yau--Tian--Donaldson (YTD) conjecture~\cite{Yau90, Tia97, Do02} which states that $X$ admits a K\"ahler-Einstein metric in $2\pi c_1(X)$ if and only if the anticanonical polarization $(X, K^{-1}_X)$ is K-polystable. There are, by now, many different proofs of this conjecture \cite{Tia15, CSW18, Zha23, BBJ21, Li22}, following the initial work of Chen–Donaldson–Sun \cite{CDS1, CDS2, CDS3} (who proved K-polystability implies existence) and Tian and Berman (who proved existence implies K-stablity)\cite{Tia97,Berman16}.

Beyond the study of K\"ahler-Einstein metrics on $(X, 2\pi c_1(X))$, other notions of canonical K\"ahler metrics have been considered. These  allow to treat cases where a K\"ahler-Einstein metric do not exist due to the classical obstructions  in terms of the automorphisms of $X$ \cite{matsushima, licherowicz, Futaki0}. Tian-Zhu~\cite{TZ1,TZ2} initiated a systematic study of the so-called \emph{K\"ahler-Ricci solitons} (KRS) on $(X, 2\pi c_1(X))$, whereas  Martelli--Sparks--Yau~\cite{MSY} developed the theory of Calabi--Yau cone structures (or, equivalently, Sasaki--Einstein structures) defined on the affine cone $K_X^{\times}$ associated to $X$. These works prompted separated investigations of the corresponding existence theories, and the formulation and proofs of appropriate modifications of the YTD conjecture in each case, see respectively \cite{DSz,CSz}. 

\bigskip
More recently, there have been developments providing  a framework to treat the existence problems mentioned above all together. The unifying geometric object is that of a \emph{$v$-soliton} K\"ahler metric $\omega \in 2\pi c_1(X)$ (called $g$-soliton in \cite{HL}), which was introduced in a foundational work by Mabuchi ~\cite{Mabuchi-multiplier}, followed by further comprehensive studies by Berman--Witt Nystr\"om~\cite{BN} and Han--Li~\cite{HL}. A $v$-soliton  is defined in terms of a fixed maximal compact torus $\T\subset \Aut(X)$ with associated canonical polytope $\Pol_X \subset \left({\rm Lie}(\T)\right)^*$,  and  a positive smooth function $v(x)>0$ on $\Pol_X$, via the equation
\[ {\rm Ric}(\omega) - \omega = \frac{1}{2}dd^c \log v(\mu_{\omega}).\]
In the above formula, $\omega$ is a $\T$-invariant K\"ahler metric in $2\pi c_1(X)$,  ${\rm Ric}(\omega)\in 2\pi c_1(X)$ is its Ricci form and $\mu_\omega : X \to \Pol_X$ is the canonically normalized $\T$-momentum map. Thus, K\"ahler--Einstein metrics correspond to $1$-solitons, KRS metrics to $e^{\ell}$-solitons~\cite{TZ1},  and Calabi--Yau cone structures on $K_X^{\times}$ to $\ell^{-(n+2)}$-solitons~\cite{AJL}, where $\ell(x)$ is a suitably defined (and in general different for each  case) affine-linear function on $\Pol_X$. As an outcome, the work \cite{HL} gives a YTD type correspondence for the existence of a $v$-soliton on $(X, \T, 2\pi c_1(X))$, expressed in terms of a suitable notion of \emph{uniform $v$-weighted Ding stability}  of $(X, \T, 2\pi c_1(X)))$ on $\T_{\C}$-equivariant test-configurations.

\bigskip
Despite the above remarkable progress, it remains a very challenging problem to use effectively the YTD correspondence in order to produce examples of $v$-solitons. A specific ramification in this general direction is the following
\begin{problem}\label{main-problem} Suppose $(X, \T, 2\pi c_1(X))$ is a smooth Fano manifold. Describe the set $\mathcal{S}(X)$ of positive smooth functions $v$ on $\Pol_X$,  such that  $(X, \T, 2\pi c_1(X))$ admits a $v$-soliton.
\end{problem}
By an idea going back to A. Futaki~\cite{Futaki0}, it turns out that $\mathcal{S}(X)\subset \mathcal{F}(X)$ where $\mathcal{F}(X)$ is the subset
\[\mathcal{F}(X):=\left\{ v\in C^{\infty}_{>0}(\Pol_X) \, | \, \int_{\Pol_X} \langle \zeta, x\rangle v(x) d\mu_{\rm DH} =0, \, \, \forall \zeta \in {\rm Lie}(\T)\right\}. \]
In the above formula,  $d\mu_{\rm DH}$ denotes the Duistermaat--Heckman measure~\cite{DH} on $\Pol_X$ induced by $2\pi c_1(X)$. Indeed, $\mathcal{F}(X)$ is identified with the set of weight functions $v$ for which the corresponding $v$-Futaki invariant $\Fut_v : {\rm Lie}(\T) \to \R$  on $X$ vanishes. Clearly,  $\mathcal{F}(X)$ is a relatively closed convex cone in $C^{\infty}_{>0}(\Pol_X)$ (in the relative $C^0(\Pol_X)$-topology).

Problem~\ref{main-problem} is particularly interesting when $\mathcal{S}(X)$ is a non-empty subset of $\mathcal{F}(X)$. We know in this case that the usual Calabi--Lichnerowicz--Matsushima obstruction vanishes, i.e. the connected component of the identity $\Aut_{\circ}(X)^{\T}$ of the group of complex automorphisms of $X$ commuting with $\T$ is reductive (see. e.g. \cite{LSW} for a general statement). In \cite{HL}, it is shown that if $X$ is a smooth \emph{toric} Fano variety, one has the equality $\mathcal{S}(X)=\mathcal{F}(X)$. Similar \emph{weight insensitive} highly symmetric Fano examples were recently found by L. Wang~\cite{W} and T. Delcroix~\cite{Del};   these works also demonstrate examples satisfying $\mathcal{S}(X)\subsetneq\mathcal{F}(X)$.

Despite the above progress,  a general understanding of $\mathcal{S}(X)$ remains an open problem. While we do not attempt in this work to obtain any deep structure result for $\mathcal{S}(X)$, we underline a basic property which provides a strong openness result \`a la LeBrun--Simanca~\cite{LS}. To state it, we introduce the space of weights
\[ \mathcal{D}(X) :=\left\{ v\in C^{\infty}_{>0}(\Pol_X) \, \Big| \, {\bf D}_v \, \textrm{is}  \, \T_\C-\textrm{coercive}\right\}, \]
where ${\bf D}_v$ denotes the $v$-weighted Ding functional (see Definitions~\ref{d:v-Ding} and \ref{d:v-Ding-coersive}).
\begin{thm}\label{thm:openedness} $\mathcal{D}(X)$ is an open convex cone in $\mathcal{F}(X)$ with respect to the relative $C^{0}(\Pol_X)$-topology.
\end{thm}
By \cite[Theorem~1.7]{HL},  $\mathcal{S}(X)=\mathcal{D}(X)$. Thus,  Theorem~\ref{thm:openedness} yields an effective estimate for the radius of an open ball in $\mathcal{S}(X)$ around a given $v_0\in \mathcal{S}(X)$, expressed in terms of the coercivity  slope of the $v_0$-weighed Ding functional, see Corollary \ref{coroQuantitative}. Our main geometric applications below use the fact that $\mathcal{S}(X)$ is relatively open in $\mathcal{F}(X)$; to obtain these results, instead of Theorem~\ref{thm:openedness} and \cite{HL},  one can merely use a LeBrun--Simanca~\cite{LS} type perturbation argument as in \cite{lahdili}, which yields the relative opennes of finite dimensional smooth families of weights in $\mathcal{S}(X)$. We present such a perturbation result adapted to our context in Appendix~\ref{a:LeBrun-Simanca}.  

{
\begin{rem} In a somewhat similar vein,  on a given K\"ahler cone $Y$ endowed with a maximal compact torus $\hat \T\subset \Aut(Y)$,  the authors study in \cite{BHLTF} the cone $\hat \tor^{\rm ext} \subset {\rm Lie}(\hat \T)$ of Sasaki--Reeb vector fields $\hat{\xi}$  such that $Y$  admits  an extremal Sasaki metric polarized by $\hat \xi$. Translated to the weighted K\"ahler terminology (see \cite{AC, ACL, lahdili}), this is equivalent to study on a given quasi-regular quotient $(X,L, \T)$ (endowed with a momentum polytope $\Pol_{L} \subset \tor^*$) the set $\tor^{\rm ext}$ of positive affine-linear functions $\ell$ on $\Pol_{L}$ such that $X$ admits a $\T$-invariant  $(\ell^{-n-1}, \ell_{\rm ext}\ell^{-n-3})$-cscK metric in $2\pi c_1(L)$. Here, $\ell$ is arbitrary positive affine linear function on $\Pol_L$ and $\ell_{\rm ext}$ is uniquely determined by $\ell$. In this setup,  it is shown in \cite{BHLTF}  that $\tor^{\rm ext}$ is not convex in general, which contrasts with Theorem~\ref{thm:openedness}.   
\end{rem}
}

\bigskip
We now discuss geometric applications of Theorem~\ref{thm:openedness}.  To this end, we consider the Sasaki geometry corresponding to the canonical bundle $K_X$ of a Fano manifold $X$.  It was shown in \cite{AJL} that the existence of a Sasaki--Einstein structure (or equivalently, the existence of a Ricci-flat K\"ahler cone structure on the affine cone $K_X^{\times}$) is equivalent to the existence of a $v$-soliton in $(X, 2\pi_1 c_1(X))$ for  a weight function $v=\ell^{-(n+2)}$ where $\ell(x)$ is a positive affine-linear function on $\Pol_X$.  

Thus,  Problem~\ref{main-problem} incorporates as a special case answering the following Conjecture due to Mabuchi and Nakagawa:
\begin{conj}\cite{MN}\label{MN-conjecture} If the smooth Fano manifold $X$ admits a KRS, then its canonical cone $Y:= K_X^\times$ admits a Calabi--Yau cone structure.
\end{conj}
We note that the recent work~\cite{HHS} shows that the above conjecture generally fails for Fano orbifolds. 

\bigskip Using Theorem~\ref{thm:openedness}, we make the following  observation related to Conjecture~\ref{MN-conjecture}.
\begin{thm}\label{thm:A} Suppose $X$ is a smooth Fano manifold which admits a K\"ahler--Ricci soliton. Then, there exists a non-negative integer $k_0$, such that for any $k\ge k_0$, the canonical cone $K^{\times}_Z$ of the Fano manifold 
\[Z:=X \times {\mathbb P}^k_{\C} \]
admits a Ricci-flat K\"ahler cone metric.
\end{thm}
Theorem~\ref{thm:A} is obtained along the following lines:  we  construct a sequence of  positive affine-linear functions $\ell_N(x)$  on $\Pol_X$,  satisfying that:  (1) $(\ell_N(x))^{-N}\in {\mathcal F}(X)$, and (2) $(\ell_N(x))^{-N}$ converges in $C^0(\Pol_X)$ to the weight $v=e^{\langle \tau, x\rangle}$  corresponding to the KRS. By Theorem~\ref{thm:openedness} and \cite[Theorem~1.7]{HL} (or equivalently by Corollary~\ref{c:Lebrun-Simanca}), for $N>>1$ there exists an $(\ell_N(x))^{-N}$-soliton on X, and hence an $(\ell_N)^{-({\rm dim}(Z) +2)}$-soliton on $Z=X \times {\mathbb P}_\C^{N-{\rm dim}(X) -2}$. The latter in turn defines a Calabi--Yau cone structure on $K^{\times}_Y$ by \cite[Prop.2]{AJL} (see also Remark~\ref{r:SE} below).

\smallskip
Our proof only shows that a sufficiently large $k_0$ as above exists,  but it does not yield a quantitative bound of $k_0$. Of course, Mabuchi--Nakagawa conjecture predicts that $k_0=0$. We also note that the $k$-dimensional complex projective space ${\mathbb P}^k_{\C}$ in Theorem~\ref{thm:A} can be replaced with any other $k$-dimensional K\"ahler--Einstein Fano manifold.

\bigskip
We next extend the geometric setup by considering, more generally, Sasaki structures on the unitary bundle $N$ in $K_X$,  which are transversal K\"ahler--Ricci solitons in the sense of Futaki--Ono--Wang~\cite{FOW}.  Equivalently, such Sasaki structures correspond to cone K\"ahler metrics on the canonical cone $K_X^{\times}$,  polarized by vector fields in the Lie algebra of $\hat\T=\T\times \Sph^1$, which give rise to transversal K\"ahler--Ricci solitons. We show in Corollary~\ref{c:v-TKRS} below that transversal KRS Sasaki structures on $N$ correspond to $v$-solitons on $X$ with $v=\ell_1^{-(n+2)}e^{\ell_2/\ell_1}$ for affine-linear functions $\ell_1>0, \ell_2$ on $\Pol_X$.  The relative openness of $\mathcal{S}(X)$ and the fact that $\ell_1$ ans $\ell_2$ are  determined by the corresponding Sasaki--Reeb vector field (see Lemma~\ref{lem:ExistFut0} below) allow us to recover a result of D.~Petrecca:
\begin{thm}\label{thm:main2}\cite{Patrecca}  Let $X$ be a smooth Fano manifold with canonical cone $K_{X}^{\times}$  and maximal compact torus $\hat \T \subset \Aut(K_X^{\times})$. Suppose $K_{X}^{\times}$ admits a compatible cone K\"ahler metric,  polarized by a Sasaki--Reeb vector field $\hat \xi_0 \in {\rm Lie}(\hat \T)$, which is a transversal K\"ahler Ricci soliton. Then, for any Sasaki--Reeb vector field $\hat \xi \in {\rm Lie}(\hat \T)$ which is sufficiently close to $\hat \xi_0$, $K_X^{\times}$ admits a compatible cone K\"ahler metric polarized by $\hat\xi$,  which is transversal K\"ahler--Ricci soliton.
\end{thm}

\bigskip We finally obtain some apriori constraints for the weight functions $v\in \mathcal{S}(X)$.

First, we give a uniform upper bound for the weight
$v(x)=(\langle \xi, x\rangle +1)^{-(n+2)}\exp\left(\frac{\langle \tau_{\xi}, x\rangle}{1+\langle \xi, x\rangle}\right)$ corresponding to a Sasaki transversal K\"ahler--Ricci soliton. 
{ \begin{thm}\label{thm:positivity} Let $X$  be a smooth Fano manifold admitting a KRS with soliton vector field $\tau$, invariant under a maximal torus $\T \subset \Aut(X)$. Then, on the canonical momentum polytope $\Pol_X$, \[ \langle \tau, x \rangle  < n.\] More generally, if $\hat \xi \in {\rm Lie}(\hat \T)$ is a Sasaki--Reeb polarization of the cone $Y=K_X^{\times}$,  normalized by the identity 
${\mathcal L}_{-J\hat \xi} \Omega = \Omega$  where $\Omega$ is the Liouville holomorphic  volume form of $K_X$,  and if, moreover, $(Y, \hat \xi)$ admits  a compatible K\"ahler cone metric which is transversal KRS, then the projected vector field $\xi \in {\rm Lie}(\T)$ and the corresponding soliton vector field $\tau_{\xi}\in {\rm Lie}(\T)$ satisfy on $\Pol_X$
\[ \langle \tau_\xi, x \rangle  < n \left(\langle \xi, x \rangle +1\right).\]
\end{thm}
}
Second, we derive a weighted version of Fujita'a inequality~\cite{fujita}:
\begin{thm}\label{thm:fujita}Let $X$  be a smooth Fano manifold admitting a $v$-soliton $\omega$. Suppose without loss of generality that $v$ is normalized  so that
\[ \int_X v(\mu_{\omega}) \omega^{[n]} = \int_X \omega^{[n]}.\]
Then the first Chern number of $X$ satisfies
\[ c_1^n(X) \le \left(\frac{n+1}{\inf_{\substack{\Pol_X}} v}\right)^{n}. \]
\end{thm}
The above inequality was established in the K\"ahler--Einstein case (i.e. with $v \equiv 1$) in a series of works by Berman-Berndtsson~\cite{BB1, BB2} under some additional conditions, and by Fujita~\cite{fujita} in general. The latter proof uses the resolution of the YTD conjecture and a computation of the Donaldson--Futaki invariant of a certain family of test configurations. Our approach here is to adapt these computations to the weighted soliton case, using the (weighted) YTD correspondence and the notions of weighted beta invariant and weighted volume  from  Han-Li~\cite{HL,HL2}.  Notice that Theorem~\ref{thm:fujita} yields an apriori constraint for $v\in \mathcal{S}(X)$ on a given Fano manifold $X$.

\bigskip The paper is organized as follows. In Sect.2, we review the theory of KRS solitons and their generalization, the $v$-solitons. In Sect.3,  we establish Theorem~\ref{thm:openedness} as an application of the results in \cite{HL}. In Sect.4,  we introduce the point of view of Sasaki geometry and the corresponding K\"ahler cones. We recall here the definition of transversal KRS soliton on the canonical cone $K_X^{\times}$ of a smooth Fano variety, and show in Corollary~\ref{c:v-TKRS} that these structures give rise to special $v$-solitons on $X$. With this observation, we apply the openness Theorem~\ref{thm:openedness} to derive the proofs of Theorems~\ref{thm:A} and \ref{thm:main2}. In Sect. 5,  we  prove Theorem~\ref{thm:positivity} and show how it can be seen as a variant of the so-called \emph{Lichnerowicz obstruction}~\cite{gaunlett-etal} which is a necessary condition for the existence of Calabi--Yau cone metrics on $K_X^{\times}$. The final Sect. 5, we prove Theorem~\ref{thm:fujita}. In the Appendix~\ref{a:LeBrun-Simanca}, we recast in the setup of $v$-solitons studied in this paper a (weaker) openness result \`a la LeBrun--Simanca from \cite{lahdili}, which gives an alternative tool for obtaining our main geometric applications, Theorems~\ref{thm:A} and \ref{thm:main2} above. In the final Appendix~\ref{a:Curvature-Identities},  we gather some well-known curvature identities for gradient Ricci solitons and use them to observe, by a simple application of the maximum principle,  that the transversal scalar curvature of a compact Sasaki transversal KRS is positive. This is a key ingredient for our proof  of Theorem~\ref{thm:positivity}.

\section{Preliminaries} 
\subsection{Fano manifolds: notation and normalization}\label{s:notation}
In what follows, $X$ will denote a smooth compact complex manifold of complex dimension $n$, for which the anti-canonical bundle $K^{-1}_X$ is ample. Such an $X$ is called a \emph{smooth Fano variety}. The Fano condition implies that $X$ is projective, and that the de Rham class $\alpha=2\pi c_1(X)=2\pi c_1(K_X^{-1})$ contains K\"ahler metrics. 

Any K\"ahler metric $\omega \in \alpha$ is de Rham cohomologous with the corresponding Ricci form $\Ric(\omega) \in 2\pi c_1(X)$, and thus we can write in this case
\[ \Ric(\omega) - \omega = \frac{1}{2} dd^c h_{\omega},\]
for a smooth function $h_{\omega}$ which (by the maximum principle) is unique up to an additive constant. Such a function will be referred to as a \emph{Ricci potential} of $\omega$;
we can further fix the additive constant  
by requiring that
\begin{equation}\label{normalized-Ricci}\int_X e^{\mathring{h}_{\omega}} \omega^{[n]} = \int_X \omega^{[n]}=:\vol(X),\end{equation}
where $\omega^{[n]}:= \omega^n/n!$ stands for the Riemannian volume form of the K\"ahler metric $\omega$. We shall  then refer to this uniquely defined Ricci potential $\mathring{h}_{\omega}$ as the \emph{normalized Ricci potential} of $\omega$.
In these terms, the \emph{K\"ahler--Einstein} condition 
\begin{equation}\label{eq:KE}
\Ric(\omega)=\omega
\end{equation}
is equivalent to $\mathring{h}_{\omega}=0$.

\bigskip
We shall next fix once for all a maximal compact real torus $\T$ inside the connected component of identity $\Aut_\circ(X)$ of the group of complex automorphisms of $X$.  The corresponding complex torus will be denoted by $\T_{\C}$. There is a canonical lift (still denoted by $\T$) of the action of $\T$ on $X$ to  an action on the canonical bundle $K_X$. The latter bundle has a further ${\Sph}^1$-extension of the lifted  $\T$ action, given  by fibre-wise multiplications with complex numbers $e^{i \theta}\in \Sph^1$. We denote by $\hat \T= \T \times \Sph^1$  the resulting $({\rm dim}(\T) +1)$-dimensional torus acting on $K_X$.
We shall respectively denote by $\tor$ and $\hat \tor$ the Lie algebras of $\T$ and $\hat{\T}$. 

We consider the space
$\mathcal{K}_{\alpha}^{\T}(X)$ of $\T$-invariant K\"ahler metrics $\omega$ on $X$, belonging to $\alpha$; by  a standard averaging argument, $\mathcal{K}_{\alpha}^{\T}(X) \neq\emptyset$. Introducing a base-point $\omega_0\in \mathcal{K}_{\alpha}^{\T}(X)$, we will identify $\mathcal{K}_{\alpha}^{\T}(X)$ with the Fr\'echet space ${\mathcal H}^{\T}_{\omega_0}(X)/\R$, where 
\[\mathcal{H}^{\T}_{\omega_0}(X):= \left\{ \varphi \in C^{\infty}(X)^{\T} \, \, | \, \, \omega_{\varphi}:= \omega_0 + dd^c\varphi >0\right\}\]
is the space of smooth $\T$-invariant K\"ahler potentials with respect to $\omega_0$.

For each $\omega \in \mathcal{K}_{\alpha}^{\T}(X)$, we let $H_{\omega}$ denote the Hermitian metric on $K_X$ whose Chern curvature is
$R^{H_{\omega}}= -i \omega$,  and by $\nabla^{H_\omega}$ the $H_{\omega}$-Chern connection on $K_X$. A basic fact in the theory is that any lift of the $\T$ action on $K_X$ gives rise to a $\T$-momentum map $\mu_{\omega}$ for $(X, \T, \omega)$, defined as follows: for any $\xi \in \tor$ and  any smooth section $s\in C^{\infty}(X, -K_X)$,
\begin{equation}\label{moment-map} {\mathcal L}_{\xi} s = \nabla^{H_\omega}_{\xi} s -i \mu_{\omega}^{\xi} s,\end{equation}
where $\mu_{\omega}^{\xi}$ satisfies
\[ \omega(\xi, \cdot) = -d\mu_{\omega}^{\xi}\] see~\cite[Proposition 8.7.2]{PGnotes}.
Therefore, the canonical lift of $\T$  to $-K_X$ gives rise to a \emph{canonically normalized} momentum map $\mu_{\omega} : X \to \tor^{*}$ whose image $\Pol_X$ is a compact convex polytope~\cite{Atiyah, GS}; one can further show (see e.g.~\cite{BN,lahdili}) that $\mu_{\omega}(\Pol_X)$ is independent of the choice of $\omega\in \mathcal{K}^{\T}_{\alpha}(X)$.  In this paper, we shall refer to $\Pol_X$ as the \emph{canonical polytope} of $(X,\T)$.

\begin{rem}\label{r:canonical_normalization}  In general, a $\T$-momentum map $\mu_{\omega} : X \to \tor^*$ is defined only up to a translation with an element of $\tor^*$;  the fact that in the Fano case there is  a canonical normalization for $\mu_{\omega}$ follows from the existence of a canonical lift of the $\T$-action on $X$ to $K_X$. 
An alternative way to define the canonical normalization for $\mu_{\omega}$ (see e.g. \cite{AJL,TZ1}) is to require that  for any $\zeta\in \tor$, the function $\mu_{\omega}^{\zeta}:= \langle \mu_{\omega}, \zeta \rangle$ satisfies
\begin{equation}\label{normalized-mu}\int_X \mu_{\omega}^{\zeta} e^{h_{\omega}} \omega^{[n]} =0, \end{equation}
where $h_{\omega}$ is any Ricci potential of $\omega$.
\end{rem}
Once we have suitably normalized $\Pol_X$, we can define the \emph{Duistermaat-Heckman}
measure $d\mu_{\rm DH}$ on $\Pol_X$ as the push-forward via $\mu_{\omega}$ of the Riemannian measure of $(X, \omega)$: for any continuous function $f$ on $\Pol_X$, we let
\begin{equation}\label{DH}
\int_{\Pol_X} f(x) d\mu_{\rm DM} := \int_X f(\mu_{\omega}) \omega^{[n]}.\end{equation}
The fact that the LHS is independent of the choice of $\omega\in \mathcal{K}^{\T}_\alpha(X)$ follows for instance from the $\T$-equivariant Moser lemma (see e.g. \cite{FM}.)

\subsection{K\"ahler--Ricci solitons}\label{s:KRsoliton} Following \cite{TZ1}, a \emph{K\"ahler Ricci soliton} (KRS for short) is a K\"ahler metric $\omega \in 2\pi c_1(X)$ which satisfies 
\begin{equation}\label{KRS}
{\rm Ric}({\omega}) - \omega = -\frac{1}{2} \cL_{J\tau} \omega,
\end{equation}
where $\tau$ is a Killing vector field for the K\"ahler structure $\omega$. In the case $\tau=0$, \eqref{KRS} reduces to the K\"ahler--Einstein condition \eqref{eq:KE}. Tian--Zhu~\cite{TZ1} have extended the Matsushima's theorem~\cite{matsushima} to the case of a KRS, which in turn yields that any K\"ahler metric satisfying \eqref{KRS} must be invariant by the action of a maximal torus in $\Aut_\circ(X)$, containing the flow of $\tau$. Up to a pull-back by an element of $\Aut_0(X)$, we can and will assume that a KRS on $X$ belongs to $\mathcal{K}^{\T}_{\alpha}(X)$ and $\tau \in \tor$. Thus, similarly to the K\"ahler--Einstein case, the KRS condition can be rewritten as 
\begin{equation}\label{KRS-h}
{h}_{\omega} = \mu_{\omega}^{\tau}\end{equation} or, equivalently,  
\begin{equation}\label{KRS-gradient}
{\rm Ric}({\omega})  -\omega = \frac{1}{2} dd^c \mu_{\omega}^{\tau}, \qquad \tau \in \tor.
\end{equation}
By Remark~\ref{r:canonical_normalization}, \eqref{KRS-h} and \eqref{DH}, if $X$ admits a KRS in $\mathcal{K}^{\T}_\alpha(X)$, then for any $\zeta\in \tor$, we have
\[\int_{\Pol_X} \langle \zeta, x \rangle e^{\langle \tau, x\rangle}d\mu_{\rm DH}=0.\]
The above condition means that $\tau$ is a critical point of the function $F: \tor \to \R$: \[ F(\zeta) := \int_{\Pol_X} e^{\langle \zeta, x\rangle}d\mu_{\rm DH}.\] Tian--Zhu~\cite{TZ1} further show that $F$ admits a unique critical point,  independent of the existence of a KRS on $X$. We shall refer to this $\tau\in \tor$ as the \emph{KRS vector field} of $(X, \T)$ and to the positive smooth function $v(x):=e^{\langle \tau, x \rangle}$ on $\tor^*$ as the \emph{KRS weight function}.

\subsection{$v$-solitons} The notion of KRS extends to the following more general geometric situation, studied by Berman--Witt Nystr\"om in \cite{BN} and,  more recently,  by Han--Li in \cite{HL}. We follow the notation of \cite[Sect.2]{AJL}.
\begin{defn}[$v$-soliton] In the setup as above, let $v(x)$ be a given positive function defined on $\Pol_X$. A K\"ahler metric $\omega \in \mathcal{K}^{\T}_{\alpha}(X)$ is called a \emph{$v$-soliton} if it satisfies
\begin{equation}\label{v-soliton}\Ric(\omega)-\omega = \frac{1}{2}dd^c \log v(\mu_{\omega}).\end{equation}
\end{defn}
Clearly, K\"ahler--Einstein metrics are $1$-solitons whereas KRS are $v=e^{\langle \tau, x\rangle}$-solitons.
Notice that if $\omega$ is a $v$-soliton it is also a $\lambda v$-soliton for any $\lambda>0$. To read off this ambiguity,  we shall sometimes consider \emph{normalized} weight functions $\mathring{v}:= \frac{v}{\int_{\Pol_X} v d\mu_{\rm DH}}$,  i.e. 
\begin{equation}\label{v-normalization} \int_{\Pol_X} \mathring{v}(x) d\mu_{\rm DH} = 1.\end{equation}
We also notice that for any $v$-soliton, $h_{\omega}=\log(v(\mu_{\omega}))$,
the linear function
\begin{equation}\label{v-Futaki} \Fut_{v} : \tor \to \R, \qquad \Fut_v(\zeta):= \int_{\Pol_X} \langle \zeta, x\rangle v(x) d\mu_{\rm DH} \end{equation}
identically vanishes by Remark~\ref{r:canonical_normalization}.
\begin{defn}[$v$-Futaki invariant] The linear function defined by \eqref{v-Futaki} is called the \emph{$v$-Futaki invariant} of $(X, \T)$. 
\end{defn}
We next define a functional ${\bf I}_{v}$ on the 
space $\mathcal{H}^{\T}_{\omega_0}(X)$ of $\T$-invariant K\"ahler potentials (see \cite{lahdili,HL}:
\[ d_{\varphi} {\bf I}_v(\dot \varphi) = \int_X v(\mu_{\omega_{\varphi}}) \dot \varphi \omega_{\varphi}^{[n]}, \qquad {\bf I}_v(0)=0.\]
Following \cite{HL}, we introduce
\begin{defn}\label{d:v-Ding}[$v$-Ding functional] The $v$-Ding functional is the map
${\bf D}_v : {\mathcal H}^{\T}_{\omega_0}(X) \to \R$ given by
\[
{\bf D}_v (\varphi) := - \left(\frac{{\bf I}_v(\varphi)}{\vol_v(X)}\right)- \frac{1}{2}\log\left(\int_{X} e^{\mathring{h}_{\omega_0}-2\varphi}\frac{\omega_0^{[n]}}{\vol(X)}\right),
\]
where we have set $\vol_v(X):=\int_{\Pol_X} v(x) d\mu_{\rm DH}= \int_X v(\mu_{\omega})\omega^{[n]}$ and $\mathring{h}_{\omega_0}$ stands for the normalized Ricci potential of the base point $\omega_0$, see \eqref{normalized-Ricci}.
\end{defn}
Notice that ${\bf D}_{v}$ does not change if we add a constant to $\varphi$, so it actually descends to a functional, denoted ${\bf D}_v(\omega_{\varphi})$, on the space $\mathcal{H}^{\T}_{\omega_0}(X)/\R \cong \mathcal{K}^{\T}_{\alpha}(X)$.

It is not hard to see that the differential of ${\bf D}_v$ is given by
\begin{equation}\label{d-Dv}
(d_{\omega_\varphi}{\bf D}_v) (\dot\varphi) = \int_X \dot{\varphi} \left( \frac{e^{\mathring{h}_{\omega_\varphi}}}{\vol(X)}-\frac{v(\mu_{\omega_\varphi})}{\vol_{v}(X)}\right)\omega^{[n]}_\varphi,
\end{equation}
so that the critical points of ${\bf D}_v$ are precisely the K\"ahler metrics $\omega_{\varphi}$ for which
\[e^{\mathring{h}_{\omega_{\varphi}}} = \vol(X)\mathring{v}(\mu_{\omega_{\varphi}})\]
i.e. the $v$-solitons.

Another consequence of the formula \eqref{d-Dv} is the following
\begin{lemma}\label{TC-invariance} The $v$-Ding functional is $\T_{\C}$-invariant, i.e. satisfies 
\[{\bf D}_v(\sigma^*(\omega)) = {\bf D}_v(\omega) \qquad \forall \sigma \in \T_{\C},\] iff the $v$-Futaki invariant $\Fut_v \equiv 0$.
\end{lemma}
\begin{proof} ${\bf D}_v$ is clearly $\T$ invariant. For any $\zeta \in \tor$, we consider the flow of $-J\zeta$, say $\sigma_t\in \T_{\C}$,   and take the derivative at $t=0$ of ${\bf D}_v(\sigma_t^*(\omega))$. By \eqref{d-Dv} 
\[\frac{d}{dt}_{|_{t=0}}{\bf D}_v(\sigma_t^*(\omega))=\int_X \mu_{\omega}^{\zeta} \left( \frac{e^{\mathring{h}_{\omega}}}{\vol(X)}-\frac{v(\mu_{\omega})}{\vol_{v}(X)}\right)\omega^{[n]}= -\Fut_v(\zeta),\]
where we have used 
\eqref{normalized-mu} for the canonically normalized momentum map $\mu_{\omega}$. The claim follows from the above by a standard argument.
\end{proof}
We end-up this section with stating one of the main results of \cite{HL}, which gives an analytic criterion for the existence of a $v$-soliton on $(X, \T)$ in terms of ${\bf D}_v$. To state it, we recall the definition~\cite{Aubin} of the Aubin  functional ${\bf J} : \mathcal{H}^{\T}_{\omega_0}(X) \to \R$:
\[ {\bf J}(\varphi) := \int_X \varphi \omega_0^{[n]}-{\bf I}_1(\varphi), \]
which descends to ${\mathcal H}^{\T}(X)/\R$, and has the property that ${\bf J}(\omega_{\varphi})\ge 0$ with ${\bf J}(\omega_{\phi})=0$ iff $\omega_{\varphi}=\omega_0$. 
\begin{defn}\label{d:v-Ding-coersive} We say that ${\bf D}_v$ is coercive  relative to $\T_{\C}$ if it is $\T_\C$-invariant  and there are positive constants $\Lambda, C$ such that 
\begin{equation}\label{e:coersive}{\bf D}_{v}(\omega) \ge \Lambda \inf_{\sigma \in \T_{\C}} {\bf J}(\sigma^*(\omega)) - C. \end{equation}
The constant $\Lambda>0$ is called the slope of coercivity of ${\bf D}_v$.
\end{defn}
\begin{thm}\cite{HL}\label{thm:HL} Let $X$ be Fano manifold,  $\T\subset\Aut_0(X)$ a maximal compact torus with canonical momentum polytope $\Pol_X\subset \tor^*$. Let $v>0$ be a positive smooth function on $\Pol_X$. Then $X$ admits a $\T$-invariant $v$-soliton in $2\pi c_1(X)$ if and only if the $v$-weighted Ding functional ${\bf D}_v$ is  coercive with respect to the complex torus $\T_{\C}$. 
\end{thm}
Versions of the above theorem have been known for KRS by the works of Cao--Tian--Zhu (see~\cite[Theorems 0.1 and 0.2]{CTZ}) and Darvas--Rubinstein (see \cite[Theorem 2.11]{DR}).

\section{A quantitative  openness result: Proof of Theorem~\ref{thm:openedness}}

We have the following elementary
\begin{prop}\label{p:comparison} Suppose $v_1, v_2$ are positive weights with respective normalizations $\mathring{v}_1$ and $\mathring{v}_2$ such that 
\[ \inf_{\Pol_X}  \left({\mathring v}_1 -{\mathring v}_2\right) = - \lambda_0, \qquad \lambda_0>0.\]
Then, there exists a constant $C_{v_1, v_2}$ such that for any $\varphi\in {\mathcal H}^{\T}_{\omega_0}(X)$
\[ \left({\bf D}_{v_1}(\omega_{\varphi}) - {\bf D}_{v_2}(\omega_{\varphi})\right) \ge -\lambda_0{\bf J}(\omega_{\varphi}) + C_{v_1,v_2}.\]
\end{prop}
\begin{proof} Using that ${\bf D}_v={\bf D}_{\mathring{v}}$, we have
\[ \begin{split}
\left({\bf D}_{v_1}(\omega_{\varphi}) - {\bf D}_{v_2}(\omega_{\varphi})\right)  &= -\left({\bf I}_{\mathring{v}_1}(\varphi) - {\bf I}_{\mathring{v}_2}(\varphi)\right) \\
&= {\bf J}_{\mathring{v}_1}(\varphi) - {\bf J}_{\mathring{v}_2}(\varphi) - \int_X \varphi \mathring{v}_1(\mu_{\omega_0})\omega_0^{[n]} + \int_X \varphi \mathring{v}_2(\mu_{\omega_0})\omega_0^{[n]},
\end{split}
\]
where ${\bf J}_v$ denotes the $v$-weighted Aubin--Mabuchi functional defined by
\[ {\bf J}_v(\varphi) := \int_X \varphi \, v(\mu_{\omega_0}) \omega_0^{[n]} - {\bf I}_v(\varphi).\]
It is easy to check from the above formula that ${\bf J}_v(\varphi)= {\bf J}_v(\omega_{\varphi})$ is a functional defined on the space of $\T$-invariant K\"ahler metrics $\omega_{\varphi} \in \alpha$, which is \emph{linear} in $v$. Another key property established in \cite{HL} is that if $v>0$ on $\Pol$, then 
\[ {\bf J}_v (\omega_{\varphi}) \ge 0.\]
The above is actually true, by continuity in $v$ (see \cite[Lemma~6.5]{AJL}),  even if we merely assume $v\ge 0$ on $P$. We thus have, by the assumption in the proposition, 
\[ {\bf J}_{\mathring{v}_1}(\varphi) - {\bf J}_{\mathring{v}_2}(\varphi) + \lambda_0 {\bf J}(\varphi)= {\bf J}_{\mathring{v}_1-{\mathring{v}_2} + \lambda_0}(\varphi) \ge 0, \]
which leads to the inequality
\[ \left({\bf D}_{v_1}(\omega_{\varphi}) - {\bf D}_{v_2}(\omega_{\varphi})\right)  \ge -\lambda_0{\bf J}(\omega_{\varphi}) + \int_X \varphi \mathring{v}_2(\mu_{\omega_0})\omega_0^{[n]}- \int_X \varphi \mathring{v}_1(\mu_{\omega_0})\omega_0^{[n]}. \]
Proposition~\ref{p:comparison} then follows from Lemma~\ref{l:comparison} below.
\end{proof}
\begin{lemma}\label{l:comparison}\cite[Lemma 13]{HL} Let $\mathring{v}$ be a normalized positive weigh-function on $\Pol_X$. Then, there exists a uniform positive constant $C_{\mathring{v}}>0$, such that for any $\varphi \in \mathcal{H}^{\T}_{\omega_0}(X)$,
\[ \sup_X \varphi - C_{\mathring{v}} \le \int_X \varphi \, \mathring{v}(\mu_{\omega_0}) \omega_0^{[n]} \le \sup_X \varphi. \]
\end{lemma}
\begin{proof} The RHS of the inequality is obvious. To obtain the LHS we observe that the inequality is invariant under translations of $\varphi$ with a constant, so  we can assume without loss that $\sup_X \varphi =0$. In this case, we need to prove
\[ \int_X \varphi \mathring{v}(\mu_{\omega_0})\omega_0^{[n]} \ge - C_{\mathring{v}}. \]
The  above  is established for instance in \cite[(27)]{DR} for $\mathring{v}=\mathring{1}=1/{\rm Vol}(X)$.  For a general weight function $\mathring{v}>0$, we let 
\[ \sup_{\Pol_X} \mathring{v} =\lambda_{\mathring{v}}>0.\]
Using $\sup_X \varphi =0$ we then have
\[ \int_X \varphi \, \mathring{v}(\mu_{\omega_0})\omega_0^{[n]} \ge \lambda_{\mathring{v}} \int_X \varphi \omega_0^{[n]} \ge- C_{\mathring{1}} \lambda_{\mathring{v}}=: -C_{\mathring{v}}.  \]
\end{proof}

Applying Theorem~\ref{thm:HL} and the previous observation we get the following quantitative openness result. 
\begin{cor}\label{coroQuantitative} Let $\Lambda_0>0$ be such that ${\bf D}_{v_0}$ is coercive with respect to $\T_{\C}$ of slope $\Lambda_0$. Then for all $v\in \mathcal{F}(X)$ such that $$ 0\leq -\inf_{\Pol_X} ( \mathring{v}-\mathring{v}_0) = \lambda_0 < \Lambda_0$$ ${\bf D}_v$ is coercive relative to $\T_{\C}$ of slope $(\Lambda-\lambda_0)$.    \end{cor}

\begin{proof}[Proof of Theorem~\ref{thm:openedness}]  Let $v_0>0$ be a weight such that ${\bf D}_{v_0}$ is coercive relative to $\T_{\C}$ of slope $\Lambda_0>0$. 
%By Definition~\ref{d:v-Ding-coersive} and Lemma~\ref{TC-invariance}, $v_0 \in \mathcal{F}(X)$, i.e. $\Fut_{v_0} \equiv 0$.  
Suppose $v>0$ is a weight such that 
\[ \big\|v- v_0\big\|_{C^0(\Pol_X)} = \varepsilon. \]
For $\varepsilon$ small enough, the above inequality yields
\[ \inf_{\Pol_X} ( \mathring{v}-\mathring{v}_0) = -\lambda_0, \qquad 0<\lambda_0<\Lambda_0.\]
We also assume  that ${\Fut}_{v}\equiv 0$. By Lemma~\ref{TC-invariance}, we know that the latter condition is equivalent to ${\bf D}_v$ being $\T_{\C}$-invariant. We let $\varphi\in {\mathcal H}^{\T}_{\omega_0}(X)$ and denote by $\omega^*_\varphi$ a K\"ahler metric in the $\T_{\C}$-orbit of $\omega_{\varphi}$, such that
\[ \inf_{\sigma \in \T_{\C}}{\bf J}(\sigma^* \omega_{\varphi})= {\bf J}(\omega^*_{\varphi}).\]
Such a metric exists by \cite[Lemma 29]{HL}.
By the argument in the proof of Lemma~\ref{l:comparison} above and the $\T_{\C}$-invariance of ${\bf D}_v$ and ${\bf D}_{v_0}$, we have
\[ \begin{split}
\left({\bf D}_v -{\bf D}_{v_0}\right)(\omega_{\varphi}) &= \left({\bf D}_v -{\bf D}_{v_0}\right)(\omega^*_{\varphi}) \\
&\ge-\lambda_0{\bf J}(\omega^*_{\varphi})- C_v= -\lambda_0 \inf_{\sigma \in \T_{\C}}{\bf J}(\sigma^*\omega_{\varphi})- C_v. \end{split}\]
 As ${\bf D}_{v_0}$ is $\Lambda_0$-coercive relative to $\T_{\C}$ and $\Lambda_0>\lambda_0$, we conclude that 
${\bf D}_v$ is also coercive relative to $\T_{\C}$. 

The convexity of $\mathcal{D}(X)$  follows from the fact that the subspace of normalized weight functions is linearly convex.   \end{proof}

%%%%%%%%%%%
%%%%%%%%%%

\section{Sasaki geometry and transversal K\"ahler--Ricci solitons.} In this section, we recall a notion of special Sasaki geometry introduced in \cite{FOW}. We use the point of view of \cite{ACL},  which will allow us to recast the possibly irregular transversal K\"ahler geometry of a Sasaki manifold in terms of the K\"ahler geometry of a given regular or quasi-regular quotient.

\subsection{Sasaki structures}\label{sss:logDisc}

We consider the following general set up:  $(N, \Ds_0, J_0)$ is a compact $(2n+1)$-dimensional strictly pseudo-convex CR manifold invariant under the action of a compact torus $\hat \T$ whose Lie algebra is denoted by $\kt$. We say that $\hat\xi \in \kt$ is a \emph{Sasaki--Reeb vector field} if $\hat \xi$ is transversal  to $\Ds_0$ and the corresponding  contact $1$-form   $\eta_0^{\hat \xi}$,  which vanishes on $\Ds$ and is equal to $1$ when evaluated at $\hat \xi$,  defines a transversal K\"ahler form $d\eta_0^{\hat \xi}$ on $(\Ds_0, J_0)$, i.e. $(d\eta_0^{\hat \xi})_{|_{\Ds_0}} >0$. We denote by $\kt_+(N) \subset \kt$ the \emph{Sasaki--Reeb} cone  of Sasaki--Reeb vector fields,  and assume that $\kt_+(N)$ is non-empty. For any $\hat \xi \in \kt_+(N)$, the data $(\hat\xi, \eta_0^{\hat \xi}, \Ds_0, J_0)$ is referred to as a \emph{Sasaki structure} on $N$.

\begin{ex}[regular Sasaki structures]\label{ex:regular}
A basic example of Sasaki manifolds are the so-called \emph{regular Sasaki} structures, described as follows.  Let $X$ be  a smooth compact complex manifold $X$ polarized by a line bundle $L$. Consider the unitary $\Sph^1$-bundle $N_{\omega_0} \subset L^{-1}$ with respect to the Hermitian metric $H_{\omega_0}$ on $L^{-1}$ whose curvature is $-i\omega_0$. Here,  $\omega_0 \in \mathcal{K}^{\T}_{2\pi c_1(L)}(X)$, where $\T$ is a maximal torus in the group of reduced automorphisms of $X$, corresponding to  a maximal torus $\hat \T \subset \Aut(X, L)$. Clearly, $\hat \T$  acts on $N_{\omega_0}$ preserving the induced CR structures $(\Ds_0, J_0)$. If $\hat\chi \in \kt$ denotes the generator of the $\Sph^1$-action on the fibres, then $\hat \chi \in \kt_+(N)$. The corresponding $1$-form $\eta_0^{\hat \chi}$ is the unique connection $1$-form on $N_{\omega_0}$ with curvature $d\eta_0^{\hat \chi}= \pi^* \omega_0$ whereas $\Ds_0$ becomes the horizontal distribution of $\eta_0^{\hat \chi}$.  The transversal K\"ahler structure in this case is just the pull-back of $\omega_0$ to $(\Ds_0, J_0)$. Note that $X= N_{\omega_0}/\Sph^1_{\hat \chi}$, which allows us to recover the K\"ahler structure $\omega_0$ on $X$ from the corresponding Sasaki structure $(\hat \chi, \eta_0^{\hat \chi}, \Ds_0, J_0)$ on $N_{\omega_0}$. This correspondence also applies to periodic Sasaki-Reeb vector fields $\hat\xi \in \kt_+(N)$, in which case $(N, \hat \xi, \Ds_0, J_0)$ is the (smooth) total space of an $\Sph^1$-orbibundle over the K\"ahler orbifold $(X:=N/\Sph^1_{\hat \xi}, \omega_0)$.  This situation is referred to in the literature as a \emph{quasi-regular} Sasaki structure.
\end{ex}
We now introduce the variation spaces of Sasaki structures: For  $\hat \xi\in\kt_+(N)$ on $(N, \Ds_0, J_0)$ and  $\eta^{\hat \xi}_0$ the associated contact form, we let
\[\Xi_{\hat \xi,\eta_0^{\hat \xi},J^{\hat \xi}}^{\hat \T}(N) := \left\{ \varphi\in C^\infty(N)^{\hat \T} \, \, | \, \, \eta_{\varphi}^{\hat \xi}:=\eta_0^{\hat \xi}+d^c_{\hat \xi} \varphi \, \, \mathrm{satisfies} \, \, d\eta_{\varphi}^{\hat \xi} >0 \mbox{ on } \Ds_{\varphi}:=\ker (\eta_{\varphi}^{\hat\xi})\right\}, \]
where $J^{\hat\xi}\in{\rm End}(TN)$ extends $J_0\in {\rm End}(\Ds_0)$ by letting $J^{\hat \xi}(\hat\xi)=0$, 
$d^c_{\hat \xi} \varphi := -d\varphi \circ J^{\hat \xi}$,  and the positivity  of $d\eta^{\hat \xi}_{\varphi}$ on $\Ds_{\varphi}$ is defined with respect to $J^{\hat \xi}_{|_{\Ds_{\varphi}}}$. We refer to \cite[\S 2]{ACL} for the basic properties of this space. In particular, for any $\varphi \in \Xi_{\hat \xi,\eta_0^{\hat \xi},J^{\hat \xi}}^{\hat \T}(N)$, $(\hat \xi, \eta_{\varphi}^{\hat \xi}, \Ds_{\varphi}, J^{\hat \xi}_{|_{\Ds_{\varphi}}})$ is a Sasaki structure and the space $\Xi_{\hat \xi,\eta_0^{\hat \xi},J^{\hat \xi}}^{\hat \T}(N)$ is called a \emph{slice of $(\hat \xi, J^{\hat \xi})$-compatible $\hat \T$-invariant Sasaki potentials}. We shall use through this paper the following identification established in \cite{ACL,HS}.
\begin{prop}\label{sasaki-identification} For any Sasaki--Reeb vector fields $\hat \chi, \hat \xi \in \kt_+(N)$, there exists a natural bijection
\[\Theta_{\hat \chi, \hat \xi} : \Xi_{\hat \xi,\eta_0^{\hat\xi},J^{\hat \xi}}^{\hat \T}(N) \cong \Xi_{\hat \chi,\eta_0^{\hat \chi},J^{\hat \chi}}^{\hat \T}(N). \]
\end{prop}
\begin{rem}\label{r:sasaki-identification}
For a (quasi) regular Sasaki structure $(N, \hat \chi, \eta_0^{\hat \chi}, \Ds_0, J_0)$ (see Example~\ref{ex:regular}), we have
\[ \Xi_{\hat\chi,\eta_0^{\hat\chi},J^{\hat\chi}}^{\hat \T}(N) \equiv \mathcal{H}^{\T}_{\omega_0}(X),\]
which, together with Proposition~\ref{sasaki-identification}, allows one to study $(\hat \xi, J^{\hat\xi})$-compatible $\hat \T$-invariant Sasaki structures on $N$ via the induced K\"ahler geometry  of $(M, \T, \alpha=2\pi c_1(L))$.
\end{rem}

\subsection{K\"ahler cones vs Sasaki structures}\label{s:Kahler-cone}
A smooth complex cone \cite{CS,ACL} $(Y,J, \hat\xi)$ is a non-compact $(n+1)$-dimensional complex manifold endowed with a free holomorphic  $\R^+$-action (sometimes denoted by $\R^+_{-J\hat \xi}$) generated by the flow of a real holomorphic vector field $-J\hat\xi$.  We shall further assume that $N= Y/\R_{-J\hat \xi}^+$ is a compact $(2n+1)$-dimensional manifold $N$. A typical example is $Y= L^{\times}$, where $L$ is a holomorphic line bundle over $X$  and  $L^{\times}$ denotes the total space of $L$ with its zero section removed. In this case, we can take $\hat \chi$ be the generator of the natural fiberwise $\Sph^1$-action on $L$, so that  $L^{\times}/\R_{-J\hat \chi}^{+}= N$ as in Example~\ref{ex:regular}. 

We shall further fix a maximal compact torus $\hat\T \subset \Aut(Y)$ and consider free $\R^+_{-J\hat \xi}$-actions for $\hat \xi \in \kt$, where $\kt = {\rm Lie}(\hat \T)$.  For any such $\hat\xi \in \kt$, we consider the space of functions
\[ \cR_{\hat \xi}^{\hat \T}(Y):= \left\{r\in C^\infty(Y, \R_{>0})^{\hat \T} \,|\, \mL_{-J\hat\xi} r =r, \, \, \hat \omega_r := {\frac{1}{4}} dd^c r^2 >0 \right\},\]
and assume that $\cR_\xi^{\hat \T}(Y)\neq \emptyset$ for at least some $\hat \xi \in \kt$; 
clearly, for such  a $\hat \xi$, any $r\in \cR_{\hat \xi}^{\hat \T}(Y)$ gives rise to a K\"ahler cone metric  $\hat \omega_r$, i.e. satisfying $\mL_{-J\hat \xi} \hat\omega_r = 2\hat\omega_r$ with respect to the $\R^+$-action generated by $-J\hat \xi$.  We shall refer to the pair $(Y, \hat \xi)$ as a \emph{polarized K\"ahler cone} and to $\hat \xi$ as a \emph{polarization}. We denote by $\kt_+(Y) \subset \kt$ the affine cone of polarizations of $Y$.

\begin{ex}[regular K\"ahler cones]
An important example is the case of \emph{regular K\"ahler cones}.  In this case,  $Y= (L^{-1})^{\times}$ for a polarization $L$ of $X$, and $\hat \chi$ is the generator of the $\Sph^1$ fiberwise action on $L^{-1}$.  In this case, the fibre-norm $r_{\omega}=|| \cdot||_{H_{\omega}}$ with respect the Hermitian metric $H_{\omega}$ on $L^{-1}$ with curvature $-i\omega$ where $\omega \in \mathcal{K}_{2\pi c_1(L)}(X)$ gives rise to a K\"ahler cone structure on $Y$ polarized by $\hat \chi$, compare with Example~\ref{ex:regular}. Furthermore, if we fix a maximal torus $\hat \T \subset \Aut_0(X,L)$, which projects to a maximal torus $\T\subset \Aut_{\rm red}(X)$ in the reduced automorphisms group of $X$, then $r_{\omega} \in {\mathcal R}^{\hat\T}_{\hat \chi}(Y)$ provided that $\omega \in \mathcal{K}^{\T}_{2\pi c_1(L)}(X)$.
\end{ex}

Any $r_0\in \cR_{\hat \xi}^{\hat \T}(Y)$ gives rise to a Sasaki structure on $N=Y/\R^{+}_{-J\hat \xi}$ as follows: 
The torus action of $\hat \T$ descends to $N$, and $\hat \xi \in \kt$ defines  a vector field on $N$. The $1$-form $\eta_0^{\hat \xi}:= d^c \log r_0$ on $Y$ is $-J\hat\xi$-basic, so we can view it  as a $1$-form on $N$; furthermore, the distribution $\Ds_0:= {\rm ker}(\eta_0^{\hat \xi}) \subset TN$ inherits a CR structure $J_0$,  defined through the almost complex structure on $\langle \hat \xi, J\hat \xi \rangle^{\perp_{\omega_{r_0}}}\subset TY$. The data $(\hat \xi, \eta_0^{\hat \xi}, \Ds_0, J_0, \hat \T)$ is a Sasaki structure on $N$, as defined in the previous subsection. Furthermore, for any other  $r\in \cR_{\hat \xi}^{\hat \T}(Y)$, the smooth function $\varphi := \log(r/r_0)$ is $(-J\hat \xi)$-invariant and defines an element of $\Xi_{\hat \xi,\eta_0^{\hat \xi},J^{\hat \xi}}^{\hat \T}(N)$. It turns out that  
\[\kt_+(Y)=\kt_+(N)=:\kt_+, \qquad  \cR_{\hat \xi}^{\hat \T}(Y) \cong \Xi_{\hat \xi,\eta_0^{\hat \xi},J^{\hat \xi}}^{\hat \T}(N), \qquad \forall \hat \xi \in \kt_+.\]
These identifications are discussed at length for example in \cite[\S 1, \S 2]{ACL}. 
In the special case when $(Y, J, \hat \chi)= (L^{-1})^{\times}$ for a polarization $L$ of $X$, our discussion here and Proposition~\ref{sasaki-identification} provide the following sequence of identifications:
\begin{equation} \label{cone-identification}
\cR_{\hat \xi}^{\hat \T}(Y) \cong \Xi_{\hat \xi,\eta_0^{\hat \xi},J^{\hat \xi}}^{\hat \T}(N) \cong \Xi_{\hat \chi,\eta_0^{\hat \chi},J^{\hat \chi}}^{\hat \T}(N) \equiv \mathcal{H}^{\T}_{\omega_0}(X), \qquad \forall \hat \xi \in \kt_+. \end{equation}

\subsection{K\"ahler cones with trivial canonical bundle}\label{sssectgammalambda} We now consider a compact Sasaki  manifold $(N, \hat \xi, \eta_0^{\hat \xi}, \Ds_0, J_0, \hat \T)$ which is {\it transversally Fano}, that is it satisfies the following additional condition: the $\hat \xi$-basic first Chern class $c_1^{\hat \xi}(\Ds_0, J_0)$ verifies 
\begin{equation}\label{Fano-Sasaki} c_1^{\hat \xi}(\Ds_0, J_0)= \lambda_{\hat \xi} [d\eta_0^{\hat \xi}]_{B,\hat \xi}, \end{equation}
for some $\lambda_{\hat \xi}>0$. This is {equivalent} to the conditions that $c_1(\Ds_0, J_0) =0$ and  the $\hat\xi$-basic first Chern class of $(\Ds_0,J_0)$ is positive (cf. \cite[Prop.4.3]{FOW}). It  follows that \eqref{Fano-Sasaki} holds true for any other Sasaki-Reeb field $\hat \chi \in \kt_+$, but the corresponding positive constants $\lambda_{\hat \xi}$ and $\lambda_{\hat \chi}$ are in general different~\footnote{If \eqref{Fano-Sasaki} holds with $\lambda_{\hat \xi}\leq 0$ then the Sasaki-Reeb cone $\kt_+$ is one dimensional.}. 

By the work of Martelli--Sparks--Yau \cite{MSY} (see also \cite[Proposition 2.5]{VC}), in the case when $\lambda_{\hat \xi}>0$ and  $\hat\xi$ is quasi-regular, the condition \eqref{Fano-Sasaki} is equivalent to the existence of a (unique up to a multiplicative constant) holomorphic volume form $\Omega$  on the corresponding polarized K\"ahler cone $(Y, \hat \xi, \hat\T)$, such that
\begin{equation}\label{CY-cones} \mathcal{L}_{-J\hat \xi} \Omega =\langle\gamma_{Y},\hat \xi \rangle \Omega, \end{equation}
where the multiplicative constant $\langle\gamma_{Y},\hat \xi \rangle \in \R_{>0}$ is canonically determined by $(Y,\hat\xi)$.  The holomorphic volume form $\Omega$ on $Y$ satisfying  \eqref{CY-cones} depends a priori on the quasi-regular Sasaki--Reeb vector field $\hat\xi$ we have chosen,  but the uniqueness yields that it can be chosen so that \eqref{CY-cones} holds for  any other quasi-regular element of $\hat\Sigma^+$. 
We can further extend the validity of \eqref{CY-cones} to irregular elements of $\hat\Sigma_+$ by density.

As the left hand side of \eqref{CY-cones} is linear with respect to $\hat \xi \in \kt$, this defines $\gamma_Y \in \kt^*$ by \footnote{In \cite[\S 2.2]{chili15a} and \cite{LiLiuXu}, the vector $\gamma_Y\in \kt^*$ appears as the restriction of the log discrepancy seen as a map on $\kt_+$, itself identified with a subspace of valuations.} \begin{equation}\label{CY-LogDisc} \langle \gamma_Y,\hat \chi \rangle := \frac{ \mathcal{L}_{-J\hat \chi} \Omega }{\Omega}.\end{equation} 

\begin{defn}\label{dfnNormSubspacesReeb} Let $(Y, \hat \xi, \hat \T)$ be a polarized K\"ahler cone which admits a global holomorphic volume form $\Omega$ satisfying \eqref{CY-cones}. We consider the sub-algebra 
\[ \tor: = \left\{ \zeta \in \hat \tor \, | \, \langle \gamma_Y,  \zeta \rangle =0\right \},\]
and the convex open polytope of {\it normalized} Reeb vector fields
\[ \hat\Sigma_+ := \left\{ \hat \xi \in \kt_+ \, | \,  \langle \gamma_Y, \hat \xi \rangle =1. \right\}. \]
\end{defn}

Both $\lambda_{\hat \xi}$ from \eqref{Fano-Sasaki} and $\langle\gamma_{Y},\hat \xi \rangle$ are linear functions on $\kt_+$ along the ray $\R_{>0} \cdot \hat\xi$ and agree at one point (the point at which we have $\lambda_{\hat \xi}=n+1$, see \cite[Proposition 2.5]{VC}). Thus, we get  
\begin{lemma} \label{lem:lambda=gamma} For any $\hat\xi\in \kt_+$, we have $\lambda_{\hat \xi}=\langle\gamma_{Y},\hat \xi \rangle$. In particular, 
\[\hat\Sigma_+ = \left\{ \hat \xi \in \kt_+ \, | \,  \eqref{Fano-Sasaki} \mbox{ holds with } \lambda_{\hat \xi} =1. \right\}.\]   
\end{lemma}

\begin{rem} { One can derive a more general relationship between the constants  $\lambda_{\hat \xi} \in \R_{>0}$ and $\gamma_Y \in \hat \tor^*$. To this end, let $\eta:= \eta_0^{\hat \xi}$ be the contact form of $\hat \xi$ on $(N, \Ds_0, J)$ and  \[ \mu^{\eta} : N \to \hat \tor^*, \qquad \langle \mu^{\eta}, a\rangle := \eta(a) \]  the induced contact $\hat \T$-momentum map on $(N, \eta)$.  Denote by $\omega=d\eta$ the transversal K\"ahler form of $(N, \Ds, J, \hat \xi)$ and by  $\rho_{\omega}$ the corresponding transversal Ricci form.
The equality \eqref{Fano-Sasaki} tells us that there exists a unique up to additive constant $\hat\T$-invariant function $h$ on $N$ such that 
\begin{equation}\label{eq:FanoCDTonc_1Basic}
 \rho_{\omega} - \lambda_{\hat \xi} \omega = \frac{1}{2}dd^c_{\hat \xi} h,
\end{equation}
where $d_{\hat \xi}^c$ is the $\hat\xi$-basic $d^c$-operator.
Contracting \eqref{eq:FanoCDTonc_1Basic} with an element of $\tor^*$ then yields 
\begin{equation}\label{eq:comparingMOMENTS}
    \gamma_{Y}= \lambda_{\hat \xi} \mu^{\eta} -\frac{1}{2} \Delta_\omega \mu^{\eta} + \frac{1}{2} d^c_{\hat\xi} h
\end{equation} where  $\Delta_{\omega}$ is the $\hat\xi$-basic Laplace operator associated to the transversal K\"ahler structure $\omega$. (Equivalently, $\Delta_{\omega}$ coincides  with the Riemannian Laplace operator  of $(N, g)$  acting on $\hat \T$-invariant smooth functions, where $g$ is the Sasaki Riemannian metric $g=\eta^2 +d\eta(\cdot,J_0\cdot)$.)} \end{rem}

\begin{ex}[The canonical cone of a Fano manifold]
The condition \eqref{Fano-Sasaki} is clearly fulfilled on the regular Sasaki manifold $(N, \hat \chi, \eta_0^{\hat \chi}, \Ds_0, J_0, \hat \T)$ of Example~\ref{ex:regular}, provided that $X$ is a Fano manifold polarized by its anti-canonical bundle, i.e. $L=K^{-1}_X$. In this case $\lambda_{\hat \chi}=1$. Notice that the corresponding K\"ahler cone is then the canonical cone $Y=K_X^{\times}$ whereas $\Omega$ is the Liouville holomorphic volume form on $K_X$. Furthermore, $\tor \subset \hat \tor$ correspond to the canonical lift of $\Aut(X)$ to $\Aut(K_X)$ and $\hat\Sigma_+$ to the dual of the momentum cone via the identification $\tor \ni \xi \mapsto \langle\xi, \cdot \rangle +1$.
\end{ex}

\subsection{Transversal K\"ahler--Ricci solitons: definition and normalization}

 %%%%%%%%%%%%%%%

\begin{defn}[see e.g. \cite{FOW}]\label{defnTRS} A (compact) Sasaki manifold $(N, \hat \xi, \eta^{\hat \xi}, \Ds, J, \hat \T)$ is a transversal KRS if there exists a $\hat\tau \in \kt$  such that
\begin{equation}\label{eq:trKRS}
 \rho_{\eta^{\hat \xi}} - \lambda_{\hat\xi} d\eta^{\hat \xi} = \frac{1}{2}dd_{\hat \xi}^c \,\eta^{\hat \xi}(\hat \tau),
\end{equation} where $\rho_{\eta^{\hat \xi}}$ is the transversal Ricci curvature of the Sasaki structure,  and $\lambda_{\hat \xi}>0$ is a positive constant.
\end{defn} 
Some immediate observations are in order. 
\begin{itemize}
 \item[(i)] For any transversal KRS the condition \eqref{Fano-Sasaki} is satisfied and  $\lambda_{\hat \xi} = \langle\gamma_Y,\hat \xi \rangle$ by Lemma~\ref{lem:lambda=gamma}.  
 
 \item[(ii)] If $(\hat \xi, \eta^{\hat \xi}, \Ds, J)$ is a transversal KRS for $\hat \tau \in \kt$ and $\lambda_{\hat\xi}>0$ it is also a transversal KRS for $\hat \tau + c\hat \xi$ and the same constant $\lambda_{\hat \xi}>0$. Thus, there exists a \emph{unique} such $\hat \tau$ satisfying the condition $\hat \tau \in \tor$ see Definition~\ref{dfnNormSubspacesReeb}. We shall denote by $\tau_{\hat \xi}$ this normalized element.
 \item[(iii)] If $(\hat \xi, \eta^{\hat \xi}, \Ds, J)$ is a transversal KRS with respect to $\tau_{\hat \xi} \in \tor$ and $\lambda_{\hat \xi}>0$, then for any $\lambda>0$
 $(\lambda\hat \xi, \eta^{\lambda\hat \xi}, \Ds, J)$ is a transversal KRS satisfying  $\tau_{\lambda \hat \xi} = \tau_{\hat \xi}$ and $\lambda_{\lambda \hat \xi}= \lambda \lambda_{\hat \xi}$. We can then assume $\hat \xi \in \hat \Sigma_{+}$ in which case $\lambda_{\hat \xi}=1$.
 
\item[(iv)] The case when $(\hat \xi, \eta^{\hat \xi}, \Ds, J)$ is a transversal KRS with $\tau_{\hat \xi} =0$ corresponds to a positive Sasaki--Einstein structure (up to scale of $\hat \xi$, see Remark~\ref{remCoeRicci} below) or, equivalently, to a Ricci-flat K\"ahler cone metric $\hat \omega = {\frac{1}{4}} dd^c r^{2}$ on $(Y, \hat \xi)$.
\end{itemize}

\begin{assumption} By (i), (ii) and (iii) above, without loss of generality we may and we shall from now on consider transversal KRS $(N, \hat \xi, \eta^{\hat \xi}, \Ds, J, \hat \T)$ such that $\hat\xi \in \Sigma_+$ and $\tau_{\hat \xi} \in \tor$, see Definition~\ref{dfnNormSubspacesReeb}.  \end{assumption}

Suppose $(\hat \xi, \eta, \Ds, J)$ is a transversal KRS with  $\hat \xi \in \Sigma_+$ and $\tau \in \tor$. In \cite{FOW}, Futaki--Ono--Wang  extended to the Sasaki setting the Tian--Zhou Futaki-type obstruction~\cite{TZ1}, by showing that 
$\tau$ satisfies
\begin{equation}\label{eq:TianZhouFutakiInvariant} \int_{N} \eta(a)\, e^{\eta(\tau)} \eta\wedge (d\eta)^n=0, \qquad \forall a\in \tor. \end{equation}
This follows from \eqref{eq:comparingMOMENTS} where we take $h= \eta(\tau)$ and observe that $$\frac{1}{2}\Delta_{\omega} \eta(a) - \mL_{-J^{\hat \xi} a} \eta(\tau) =  \frac{1}{2}\Delta_{\omega} \eta(a) - \mL_{-J^{\hat \xi }\tau} \eta(a)$$ is self-dual with respect to the volume form $e^{\eta(\tau)} \eta\wedge (d\eta)^n$.

By the arguments of \cite{TZ1,FOW} and \cite{MSY}, we have
\begin{lemma} \label{lem:ExistFut0} Let $(N, \hat \chi, \Ds_0, J_0, \hat \T)$ be a compact Sasaki manifold satisfying \eqref{Fano-Sasaki}. Then
\begin{itemize}
 \item  For any $\hat \xi \in \hat \Sigma_+$ there exists a unique $\tau_{\hat \xi} \in \tor$ such that \eqref{eq:TianZhouFutakiInvariant} holds.
 \item There exists  a unique $\hat \xi \in \hat \Sigma_+$ such that $\tau_{\hat \xi}=0$.
\end{itemize}
\end{lemma}
\begin{proof}
For any $\hat \xi \in \hat \Sigma_+$, pick $\eta=\eta^{\hat \xi}_{\varphi}$ for $ \varphi \in \Xi(\hat\xi,\eta_0,J^{\hat \xi})^{\hat \T}$ the functional $V_{\hat \xi}: \tor \ra \R$ defined by
\begin{equation}\label{eq:TianZhouFutakiInvariant_funct}
 V_{\hat \xi}(a):= \int_{N} \, e^{\eta(a)} \eta\wedge (d\eta)^n, \quad a \in \tor,
\end{equation} is convex and does not depend on the choice of $\eta$ within $\Xi(\hat \xi,\eta_0^{\hat \xi},J^{\hat \xi})^{\hat \T}$. 
Therefore, there exists a unique $\tau \in \tor$, the critical point of $V_\xi$, satisfying \eqref{eq:TianZhouFutakiInvariant}.

For the second part, by \cite{MSY}, there exists a unique $\hat \xi_0 \in \Sigma_+$ such that the corresponding contact $1$-form $\eta_0$ satisfies
\[ \int_N \eta_0(a) \eta_0 \wedge (d\eta_0)^n =0 \qquad  \forall a\in \tor.\]
This in turn is equivalent to \eqref{eq:TianZhouFutakiInvariant} for $\hat \xi \in \Sigma_+$ and $\tau_{\hat \xi} =0$.
\end{proof}
Similarly, we can introduce transversal KRS K\"ahler cones.
\begin{defn} We say that 
a polarized K\"ahler cone $(Y,J,\hat\xi)$ admits a \emph{compatible transversal KRS cone K\"ahler metric} $\hat \omega = \frac{1}{4}dd^c r^2$ with $r\in \mathcal{R}_{\hat{\xi}}(Y)$ 
if the associated Sasaki structure on $N= r^{-1}(1)$ (see \S\ref{s:Kahler-cone}) is a transversal KRS. 
\end{defn}

\begin{rem}\label{remCoeRicci} { A polarized K\"ahler cone $(Y,J, \hat \xi, \hat \omega)$ with  $\hat\omega = \frac{1}{4}dd^cr^2$  for a radial potential function  $r\in\mathcal{R}_{\hat\xi}(Y)$ is K\"ahler Ricci-flat if and only if 
the associated Sasaki manifold $(N, \hat \xi, \eta^{\hat \xi}, \Ds, J, \hat \T)$ (see \S\ref{s:Kahler-cone}) is transversally K\"ahler--Einstein with transversal scalar curvature equal to $2n(n+1)$.
To see this,  let $\omega=d\eta^{\hat \xi}$ and $\rho_{\omega}$  denote the corresponding transversal K\"ahler and Ricci forms on the Sasaki manifold $N=r^{-1}(1)$. Then we have 
(see e.g.~\cite{MSY})  \begin{equation}\label{eqConeRiccivsSasaki}\rho_{\hat{\omega}} = \rho_{\omega} -(n+1)\omega.\end{equation}
 } \end{rem}

\subsection{Transversal KRS as $v$-solitons} 
We start with a technical result obtained in the following general set up: $(N, \Ds, J)$ is a fixed strictly pseudo-convex CR manifold and $(\hat\xi, \hat\chi)$ are commuting Sasaki--Reeb vector fields.
We denote by $\omega_{\hat \xi}:= d\eta^{\hat{\xi}}$ and $\omega_{\hat \chi}:= d\eta^{\hat \chi}$ the corresponding transversal K\"ahler forms and by $g_{\hat \xi}$ and $g_{\hat\chi}$ the corresponding transversal Riemannian metrics. We further suppose that $(N, \Ds, J, \hat \xi)$ (with $\hat \xi \in \hat \Sigma_+$) is transversally Fano, i.e. the transversal Ricci form $ \rho_{\omega_{\hat \xi}}$ satisfies (see \eqref{eq:FanoCDTonc_1Basic})
\[ \rho_{\omega_{\hat\xi}} -  \omega_{\hat\xi} = \frac{1}{2} dd^c_{\hat\xi} h, \]
where $h$ is a $(\hat \xi, \hat\chi)$-invariant smooth function.
\begin{lemma} Let  $\lambda_{\hat\chi}$ be the constant introduced in \S \ref{sssectgammalambda}. Then the transversal Ricci form $\rho_{\omega_{\hat \chi}}$ of the transversal K\"ahler structure $(g_{\hat\chi}, \omega_{\hat \chi})$ satisfies
\[\rho_{\omega_{\hat \chi}} - \lambda_{\hat \chi} \omega_{\hat \chi} = \frac{1}{2} dd^c_{\hat\chi}  h - \frac{(n+2)}{2} dd^c_{\hat \chi} \log f, \qquad f:=\eta^{\hat\chi}(\hat\xi). \]
\end{lemma} 
\begin{proof} Let $f:= \eta^{\hat\chi}(\hat\xi)$ be the positive Killing potential of $(g_{\hat\chi}, \omega_{\hat\chi})$. We use the following general relation between the transversal Ricci curvatures (see e.g. \cite{JL}):
\begin{equation}\label{Ricci-CR}
 \rho_{\omega_{\hat\xi}} - \rho_{\omega_{\hat\chi}}  = \frac{(n+2)}{2} dd^c_{\hat\chi} \log f - \frac{1}{2}d\left(\Big(\frac{\Delta_{g_{\hat\chi}} f}{f} +(n+2) \frac{|df|^2_{g_{\hat\chi}}}{f^2}\Big)\eta^{\hat\chi}\right). \end{equation}
Furthermore, using that $h$ is $\hat\chi$-invariant, $\eta^{\hat\chi} = f \eta^{\hat\xi}$,  and  (see \cite[Lemma 1]{AC}) 
\[\hat\xi = f \hat\chi  - \omega_{\hat\chi}^{-1}(df_{|_{\Ds}}), \]
we get
\begin{equation}\label{dc}
 d^c_{\hat\xi} h   = d^c_{\hat\chi} h  - \left\langle dh, df \right\rangle_{g_{\hat\chi}} \eta^{\hat\chi}. \end{equation}
 Finally, using \eqref{Ricci-CR}, \eqref{dc} and
 \[\omega_{\hat\xi} = d\eta^{\hat\xi}, \qquad \omega^{\hat\chi} = d\eta^{\hat\chi}, \qquad \eta_{\hat\xi}= f \eta^{\hat\chi}, \]
 we get
 \[
 \begin{split}
0  = & \rho_{\omega_{\hat\xi}} - \omega_{\hat\xi} - \frac{1}{2} dd^c_{\hat\xi} h  \\
     =  & \rho_{\omega_{\hat\chi}} + \frac{(n+2)}{2} dd^c_{\hat\chi} \log f  - \frac{1}{2}dd^c_{\hat\chi} h  \\
          & -d\left[f + \frac{1}{2}\left(\frac{\Delta_{g_{\hat\chi}} f}{f} +(n+2) \frac{|df|^2_{g_{\hat\chi}}}{f^2}\right) - \frac{1}{2}\left\langle dh, df \right\rangle_{g_{\hat\chi}}\right] \eta^{\hat\chi}
          \end{split} \]
Evaluating the latter on the vector field $\hat\chi$, we get that
  \[  \left[f + \frac{1}{2}\left(\frac{\Delta_{g_{\hat\chi}} f}{f} +(n+2) \frac{|df|^2_{g_{\hat\chi}}}{f^2}\right) - \frac{1}{2}\left\langle dh, df \right\rangle_{g_{\hat\chi}}\right]=\lambda\]
  is a constant. The first claim follows easily from the above.  This constant must be positive when $N$ is compact (which follows by considering a point of maximum of $f>0$) and is then also determined from the basic cohomology of $\hat \chi$ by the relation $[c_1(\Ds)]_{B,\hat\chi} = \lambda [\omega_{\hat\chi}]$. Thus, $\lambda =\lambda_{\hat \chi}$, see \S \ref{sssectgammalambda}.  \end{proof}
  We thus get the following generalization of \cite[Prop.2]{AJL}.
\begin{cor}\label{c:v-TKRS} Suppose $\kt$ is a maximal abelian subalgebra of $\mathfrak{cr}(N, J, \Ds)$,  and $\hat\xi, \hat \chi \in \kt$. Then $(\hat \xi, \Ds, J)$ is a transversal KRS  with soliton vector  $\tau_{\hat \xi} \in \tor$  iff $(\hat \chi, \Ds, J)$ is a transversal $v_{\hat \xi}$-soliton for the weigh function
\begin{equation}\label{eq:tKRSweight}
 v_{\hat \xi} : = e^{\big(\frac{\eta^{\hat \chi}( \tau_{\hat \xi})}{\eta^{\hat \chi}(\hat \xi)}\big)} \left(\eta^{\hat \chi}(\hat \xi)\right)^{-(n+2)}. 
\end{equation}
\end{cor}

\subsection{Transversal KRS's from a fixed regular quotient} By virtue of Corollary~\ref{c:v-TKRS}, 
we now suppose $(X, \T)$ is a Fano manifold with $\T\subset {\rm Aut}_{r}(X)$ a maximal torus with canonically normalized momentum polytope $\Pol_X$.   Any $\T$-invariant K\"ahler metric $\omega\in 2\pi c_1(X)$ can be viewed as the reduction of a Sasaki manifold $(N, J, \Ds)$ by a regular Sasaki Reeb vector field $\hat \chi \in \kt$, where $\hat T= \Sph^1_{\hat \chi}\times \T$ is defined  on $N$ via the natural lift of $\T$ to $K_X$. Furthermore, the Lie algebra $\tor = {\rm Lie}(\T)$ is naturally identified with the  subspace $\tor \subset \kt$  defined in Definition~\ref{dfnNormSubspacesReeb}. 

The polytope $\Sigma_+ \subset \hat \tor_+$ can be equivalently characterized in terms of $(X, \T, \omega)$ as the set of elements $\zeta\in\mathfrak{t}$ such that the affine linear function $\ell_\zeta(x):=\langle x,\zeta\rangle+1$ is strictly positive on $\Pol_X$. Indeed, any affine-linear function $\ell_{\zeta}$ as above gives rise to an element  $\hat \zeta \in \Sigma_+$ given by the horizontal lift of $\zeta$ to $\Ds$ plus $\ell_{\zeta}(\mu_{\omega})\hat \chi$; conversely,  the quantity $\eta_{\hat \chi}(\hat \xi)$ defines an affine-linear function on $\Pol_X$ of the form $\langle \zeta, x \rangle +1$. Thus, we have an identification of $\Sigma_+$ with the dual (open) polytope $\mathring{\Pol}_X^*$ of $\Pol_X$:
\begin{equation}\label{eq:dual-pol} \Sigma_+ \cong \mathring{\Pol}_X^*,  \qquad \ \mathring{\Pol}_X^*:=\left\{ \zeta \in \tor \, | \, \ell_{\zeta}(x):=\langle \zeta, x \rangle + 1 >0  \, \, \forall x \in \Pol_X\right\}.\end{equation}
By Lemma~\ref{lem:ExistFut0}, there exists a {smooth} function 
\[ \tau: \mathring{\Pol}_X^* \to\tor\] 
such that
\begin{equation}\label{l:Fut=0} \forall \xi\in \mathring{\Pol}_X^*, \, \,  \, v_\xi:=\exp\left(\frac{\langle\mu,\tau(\xi)\rangle}{\ell_\xi}\right)(\ell_\xi)^{-(n+2)} \, \,  \textrm{satisfies} \,  \, {\rm Fut}_{v_\xi}\equiv 0, \end{equation}
where $\Fut_v$ is the weighted Futaki invariant defined in \eqref{v-Futaki}.

\begin{rem}\label{r:SE} A special case of Lemma~\ref{lem:ExistFut0} appears when $\xi_0 \in \mathring{\Pol}_X^*$ is such that $\tau(\xi_0)=0$. Equivalently,  $v_{\xi_0} = \ell_{\xi_0}^{-(n+2)}$ where $\xi_0$ is uniquely determined (see \cite{MSY}) by the property  ${\Fut}_{\ell_{\xi_0}^{-(n+2)}}\equiv 0$ or,  equivalently, 
\[
(dV_{\xi_0})_0 (\dot{\tau})=\int_X \langle\mu_\omega,\dot{\tau}\rangle \frac{\omega^{[n]}}{\ell_{\xi_0}(\mu_\omega)^{n+2}}={\rm Fut}_{\ell_{\xi_0}^{-(n+2)}}(\dot{\tau})= 0.
\] 
The corresponding $v_{\xi_0}$-soliton (if it exists) then corresponds, up to a transversal homothety (see Remark~\ref{remCoeRicci}), to a Calabi--Yau K\"ahler cone structure on $K_X^{\times}$, compare with \cite[Prop.2]{AJL}
\end{rem}

\subsection{From K\"ahler--Ricci solitons  to Calabi--Yau cones: proof on Theorem~\ref{thm:A}}

\begin{proof}[Proof of Theorem~\ref{thm:A}]  Let $(X, \T)$ be a smooth Fano variety admitting a KRS. We consider the sequence of rational functions on $\tor\times \tor^*$ 
\[ q_N(\xi, x):=\left(1-\frac{\langle x,\xi\rangle}{N}\right)^{-N}, \] which for $N>>1$  is well-defined and converges uniformly on a compact subset of $\tor \times \tor^*$ to the function $e^{\langle\mu,\xi\rangle}$.

Let $\mathring{\Pol}_X^* \subset \tor$ denote the dual  (open) polytope of $\Pol_X$, see \eqref{eq:dual-pol}. Thus, for any $\xi \in -N{\mathring{\Pol}_X^*}$, $q_N(\xi, x)$ is a positive function on $\Pol_X$ and we consider the volume functional
\[V_N (\xi) : = \int_X \left(1-\frac{\langle \mu_{\omega},\xi\rangle}{N}\right)^{-N+1} \omega^{[n]}, \qquad \omega \in {\mathcal K}^{\T}_{2\pi c_1(X)}\]
which is well-defined, convex and proper on {$-N\mathring{\Pol}^*_X$}. Let $\xi_N \in -N\mathring{\Pol}^*_X$ be the unique minimizer of $V_N$: $\xi_N$ then satisfies 
\begin{equation} \label{N-Futaki}\int_X \mu^{\zeta}_{\omega}\left(1-\frac{\langle \mu_{\omega},\xi_N\rangle}{N}\right)^{-N} \omega^{[n]}=0, \qquad \forall \zeta \in \tor, \, \forall \omega \in {\mathcal K}^{\T}_{2\pi c_1(X)}.\end{equation}
i.e.  $v_N(x):= q_N(\xi_N, x)>0$ on $\Pol_X$ and $\Fut_{v_N}(\zeta) =0$ for any $\zeta \in \tor$.

\smallskip
We now show that if $\tau\in \tor$ is the Tian--Zhu KRS vector field (note that $(X, \T)$ admits an $e^{\langle \tau, x\rangle}$-soliton by the hypothesis) then 
\[\lim_{N\to \infty} \xi_N = \tau.\]
To see this, we can use  the implicit function theorem. Indeed, consider the function 
\[F(a,t):= e^{\left(\frac{a+1}{a}\right)\log(1+at)}, \qquad (a,t)\in \R^2: \,  \, \, 0<1+at< 2, \]
which admits an analytic extension at $a=0$ with $F(0,t)=e^t$. Notice that $F(-\frac{1}{N}, t)=(1 -\frac{t}{N})^{-N+1}$ and,  for any fixed $a$, the function $t\to F(a, t)$ is strictly convex  as soon as $a>-1$. Consider the function
\[
W(a, \xi) := \int_X F(a, \langle \xi, \mu_{\omega}\rangle) \omega^{[n]},\]
defined on the domain
\[U :=\left\{(a, \xi) \in \R \times \tor \, | \, a>-1, \, \,  -1<a\langle \xi, x \rangle < 1  \, \, \forall x \in \Pol_X\right\}. \]
By the properties of $F(a,t)$ mentioned above,  $W(-\frac{1}{N}, \xi)= V_N(\xi)$ and, for any fixed $a>-1$,  the map $\xi \to W(a, \xi)$ is strictly convex on the domain \[ U_a:=\{ \xi \in \tor \, \,  | \, \, -1<a\langle \xi, x \rangle < 1\}.\]
Applying the implicit function theorem to 
\[ \Psi : U \to \tor^*, \qquad \Psi(a, \xi):= (d_{\xi} W)(a, \xi), \]
there exists a smooth path $(a,\tau_a)\in U$, defined for $|a|<\varepsilon$, such that $\Psi(a,\tau_a)=0$ and $\tau_0 =\tau$. Clearly, for $N>>1$, $\xi_N = \tau_{-\frac{1}{N}}$ by the uniqueness of the critical point, and whence $\lim_{N\to \infty} \xi_N = \tau$.

\smallskip
Using $\lim_{N\to \infty} \xi_N = \tau$, we have $\lim_{N\to \infty} v_N(x) = e^{\langle \tau, x \rangle}$ in $C^{0}(\Pol_X)$.  As $X$ admits an $e^{\langle \tau, x \rangle}$-soliton by assumption,   and $v_N(x)$ tends uniformly on $\Pol_X$ to  $e^{\langle \tau, x \rangle}$ satisfying $\Fut_{v_N}=0$, we can apply Theorems~\ref{thm:openedness}  and \ref{thm:HL} to conclude that there are $v_N$-solitons for any $N>>1$. Alternatively, we can appeal to Corollary~\ref{c:Lebrun-Simanca} applied to the smooth family of weights $v_a(x):= e^{\frac{1}{a}{\log}(1 + a\langle \tau_a, x\rangle)}\in \mathcal{F}(X)$.

Taking the product of a $v_N$-soliton on $X$ with the Fubini--Study metric on ${\mathbb P}_\C^{N-n-2}$ gives rise to a $v_N$-soliton 
on $Z=X \times \mathbb{P}_{\C}^{N-n -2}$. As {$Z$} is $(N-2)$-dimensional,  this in turn defines a Calabi--Yau cone structure on $K_Z^{\times}$ by Remark~\ref{r:SE}.
\end{proof}

\subsection{Opennes of transversal KRS: Proof of Theorem~\ref{thm:main2}}
\begin{proof}[Proof of Theorem~\ref{thm:main2}] We consider the smooth family of weights $v_{\xi}$, $\xi \in \mathring{\Pol}_X^*$ defined by \eqref{l:Fut=0}. By Corollary~\ref{c:v-TKRS}, Proposition~\ref{sasaki-identification} and Remark~\ref{r:sasaki-identification}, the existence of a $\T$-invariant $v_{\xi}$-soliton in $2\pi c_1(X)$ is equivalent with the existence of a transversal KRS soliton  in $\Xi_{\hat \xi,\eta_0^{\hat\xi},J^{\hat \xi}}^{\hat \T}(N)$. By the discussion in Section~\ref{s:Kahler-cone}, this is also equivalent to the existence of a $\hat \xi$-polarized transversal KRS cone metric on $Y=K_X^{\times}$.  Theorem~\ref{thm:main2} then follows from  Theorems~\ref{thm:openedness} and \ref{thm:HL}, or by Corollary~\ref{c:Lebrun-Simanca} applied to $v_{\xi}$.
\end{proof}
\begin{rem} The assumption in the above proof that $K^{\times}_X$ is a cone of \emph{smooth} Fano manifold, or,  equivalently, that the Sasaki manifold  $(N, \Ds, J, \hat\T)$ admits a regular Sasaki-Reeb vector field can be removed by considering instead a quasi-regular Sasaki--Reeb field $\hat \chi \in \hat \tor_+$. In this case, the quotient $X:= N/\Sph^1_{\hat \chi}$ is a K\"ahler Fano orbifold. We can still apply Corollary~\ref{c:v-TKRS} which hold for orbifolds. Instead of Theorem~\ref{thm:openedness}, which uses the results in \cite{HL} on a smooth Fano manifold,  one can use an openness result a la LeBrun--Simanca (which we recall in Appendix~\ref{a:LeBrun-Simanca} and whose proof can be adapted to orbifolds) in order to establish  the existence of a transversal KRS in $\Xi_{\hat \xi,\eta_0^{\hat\xi},J^{\hat \xi}}^{\hat \T}(N)$ for all Sasaki--Reeb vector fields near $\hat \xi_0$. This is closer to the original approach in \cite{Patrecca}. \end{rem}

\section{A Lichnerowicz type obstruction for transversal KRS soliton: Proof of Theorem~\ref{thm:positivity}}
We discuss in this subsection a necessary condition in terms of the function $\tau(\xi)$ for the weight function $v_{\xi}$ to belong to $\mathcal{S}$, i.e  $K_X^{\times}$ to admit a compatible transversal KRS polarized by $\hat \xi$.

\subsection{Proof of Theorem~\ref{thm:positivity}}
\begin{prop}\label{p:strong-positivity} 
 Let $(Y, \hat\xi)$ be a  smooth polarized complex cone endowed with a holomorphic volume form $\Omega$, satisfying ${\mathcal L}_{-J\hat \xi} \Omega = \Omega$. Let $\hat \T$ be a maximal torus in $\Aut(Y)$ such that $\hat \xi \in \hat \tor={\rm Lie}(\hat \T)$. Suppose that $(Y, \hat \xi)$ admits a compatible transversal K\"ahler--Ricci soliton, $\hat \omega = \frac{1}{4}dd^c r$, $r\in \mathcal{R}_{\hat \xi}^{\hat \T}(Y)$,   with corresponding soliton vector field 
 $\tau_{\hat \xi}\in \hat \tor$ normalized by ${\mathcal L}_{-J\tau_{\xi}}\Omega=0$. Then $$\hat\xi-\frac{1}{n}\tau_{\hat \xi}$$ is a Sasaki--Reeb polarization on $Y$. 
\end{prop}
\begin{proof} In the notation of the previous subsections, we have  $\hat \xi \in \hat\Sigma_+$ and $\tau_{\hat \xi}\in \tor$. We shall work on the corresponding Sasaki manifold $(N:=r^{-1}(1), \Ds, J, \hat\xi, \eta^{\hat \xi})$, which is  a transversal KRS by assumption. We want to show that $\hat\xi-\frac{1}{n}\tau_{\hat \xi} \in \hat \tor_+(N)$ or, equivalently, (see \cite{CS,BV,ACL})
\begin{equation}\label{inequality} \eta^{\hat \xi}\left(\hat\xi-\frac{1}{n}\tau_{\hat \xi}\right)=1 -\frac{1}{n} \eta^{\hat \xi}(\tau_{\hat \xi})>0.
\end{equation}
Contracting the transversal KRS equation \eqref{eq:trKRS} with the transversal K\"ahler form  $\omega:=d\eta^{\hat \xi}$ gives for the transversal scalar curvature $\Scal(\omega)$
\[ \Scal(\omega) = 2n  - \Delta_{\omega} f, \qquad f:= \eta^{\hat \xi}(\tau_{\hat \xi}).\]
By \eqref{eq:comparingMOMENTS} (with $h=f$ and contracting $\tau_{\hat \xi}$), we also get
\begin{equation}\label{eqLapNormKRS}
f-\frac{1}{2}\Delta_{\omega}f +\frac{1}{2}\langle d^c_{\hat \xi} f , \tau_{{\hat \xi}}\rangle =0,\end{equation}
so we obtain
\[ \Scal(\omega) = 2n - 2f - \langle d^c_{\hat \xi} f, \tau_{\hat \xi}\rangle. \]
As $\Scal(\omega)>0$ everywhere on $N$ by Corollary~\ref{c:scal>0} established in the Appendix~\ref{a:Curvature-Identities}, evaluating the above equality at a point of global maximum of $f$ 
yields  \eqref{inequality}.  \end{proof}

\begin{proof}[{ Proof of Theorem~\ref{thm:positivity}}] We can realize $X$ as the regular K\"ahler quotient of a polarized cone $(Y=K_X^{\times}, \hat \xi)$ endowed with a compatible transversal K\"ahler--Ricci soliton with vector field $\hat \tau$. The lifts $\hat \xi$ and $\hat \tau$ of the vector fields $0$ and $\tau$ on $X$  are normalized respectively by  ${\mathcal L}_{-J\hat \xi} \Omega =\Omega$ and $ {\mathcal L}_{-J\hat \tau} \Omega =0$. Thus,  by 
Proposition~\ref{p:strong-positivity},  the vector field $n\hat \xi - \hat \tau$ defines a Sasaki--Reeb polarization on $Y$. In terms of $X$, this means that  $n - \mu_{\omega}^{\tau} >0$. The second statement follows similarly.
\end{proof}

\subsection{Theorem~\ref{thm:positivity} as a Lichnerowicz type condition}
It is well-known that the normalized characters for the linear action of $\hat \T$ on the spaces $H^0(X, -mK_X)$  of $m$-plurianticanonical sections of $X$ are rational lattice points inside $\Pol_X$ (see e.g. \cite[Lemma 13]{lahdili}). From this point of view, the first inequality in Theorem~\ref{thm:positivity} yields an a priori bound on the spectrum of the operator $\mathcal{L}_{\tau}$ acting on $H^0(X, -mK_X)$: 
\begin{cor}\label{c:plurianticanonical} If $(X, \T)$ admits a $\T$-invariant KRS with soliton vector field $\tau$ and $s\in H^0(X, -mK_X)$ is a non-trivial $m$-plurianticanonical holomorphic section such that $\frac{1}{m}\mathcal{L}_{\tau} s = i\lambda s, \, \lambda \in \R$, then $(n-\lambda)>0$. 
\end{cor}
On the other hand, plurianticanonical sections of $X$ give rise to certain holomorphic functions on the corresponding cone $Y = K_X^{\times}$.  Thus, Corollary~\ref{c:plurianticanonical} should also lead to an a priori bound on the eigenvalues (also called \emph{charges} in \cite{gaunlett-etal}) for the infinitesimal action of the soliton vector field $\tau$ on the space of holomorphic functions on $Y$. This is reminiscent to but different from the bounds on the charges of holomorphic functions of $Y$ under the action of the Reeb field of a Calabi--Yau cone structure on $Y$,  obtained in \cite{gaunlett-etal}.

We provide below a direct  argument for this bound, independent of the proof of Theorem~\ref{thm:positivity}.

\begin{lemma}\label{lem:LichBound}
Let $(Y,\hat \xi)$ be a smooth polarized {complex cone} satisfying the conditions of Proposition~\ref{p:strong-positivity} with respect to a maximal torus of automorphisms $\hat \T \subset \Aut(Y)$.  If $\varphi : Y \ra \C$ is  a non-trivial holomorphic function on $Y$ satisfying $\cL_{\hat \xi} \varphi = i \kappa \varphi$ and $\cL_{\tau_{\hat \xi}} \varphi  =  i\kappa\lambda \varphi$ for some $\lambda\in \R$ and $\kappa \in \R_{>0}$, then  
$$(n - \lambda)>0.$$ \end{lemma}
\begin{proof}  Let $r\in \cR_{\hat \xi}^{\hat \T}(Y)$ be the potential of a compatible transversal KRS cone metric on $(Y, \hat \xi)$,  with soliton vector field  $\tau=\tau_{\hat\xi} \in \tor$. As in the proof of Proposition~\ref{p:strong-positivity},  we shall work on the link $N:=r^{-1}(1)$, with the induced transversal KRS Sasaki structure $(\Ds, J, \hat \xi, \eta^{\hat \xi})$.  We still denote by $\omega:=d\eta^{\hat \xi}$ the transversal K\"ahler structure and let 
$f :=\eta^{\hat \xi}(\tau_{\hat \xi})= \mu_{\omega}^{\tau}$ be the transversal Killing potential of $\tau_{\hat \xi}$.  

We will first prove that,  when restricted on $N$, $\varphi$ satisfies the inequality
\begin{equation}\label{eqCoroWeitzenbockCone}
    \Delta_{\omega,f}(e^{-\kappa f}|\varphi|^2) \le  2\kappa(n -\lambda)e^{-\kappa f}|\varphi|^2, %-  ||d (e^{-\kappa f}\varphi)||^2_\omega,
\end{equation} where $\Delta_{\omega,f} \psi := \Delta_\omega \psi  + \langle d\psi, df \rangle_{\omega}$ is the twisted $\hat\xi$-basic Laplacian (recall that $\Delta_\omega$ denotes the $\hat\xi$-basic Laplacian).

As $\varphi$ is holomorphic, the assumptions equivalently read as
\begin{equation*} {\mathcal L}_{-J\hat \xi} \varphi = \kappa \varphi, \qquad \mathcal{L}_{-J\tau_{\hat \xi}} \varphi = \kappa\lambda \varphi,
\end{equation*}
and hence 
\begin{equation}\label{all-weights} {\mathcal L}_{\hat \xi}|\varphi|^2=0, \qquad \mathcal{L}_{-J\hat \xi} |\varphi|^2 = 2\kappa  |\varphi|^2, \qquad \mathcal{L}_{-J\tau_{\hat \xi}} |\varphi|^2 = 2\kappa\lambda |\varphi|^2.\end{equation}
At any point of $N$, we can write (see e.g. \cite[Lemma 1]{AC}) 
\begin{equation}\label{decompose}
\tau_{\hat \xi} = f\hat \xi -\omega^{-1}(df), \qquad  J\tau_{\hat \xi} = f J\hat \xi - g_{\omega}^{-1} (df).\end{equation}
We get from the above
\begin{equation}\label{intermediate} 
\langle d|\varphi|^2, df \rangle_{\omega}= 2\kappa(\lambda -f)|\varphi|^2,\end{equation}
where $\langle \cdot, \cdot \rangle_{\omega}$ stands for the inner product induced by the transversal K\"ahler structure $\omega$  on $\hat\xi$-basic tensors.  Together with  \eqref{eqLapNormKRS}, the latter yields 
\begin{equation*}\label{eqTwistedterm}
    \left\langle d (e^{-\kappa f}|\varphi|^2), df \right\rangle_{\omega} = 2\kappa e^{-\kappa f}|\varphi|^2\left(\lambda - \frac{1}{2}\Delta_{\omega} f\right).
\end{equation*} 
A direct computation using \eqref{intermediate} shows that 
\begin{equation*}\label{eqLapProofIneq1}
      \Delta_{\omega} (e^{-\kappa f}|\varphi|^2) =  e^{-\kappa f}|\varphi|^2\left( -\kappa\Delta_\omega f  -\kappa^2||df||^2_\omega +4\kappa^2 \lambda - 4\kappa^2 f \right) + e^{-\kappa f}  \Delta_{\omega}|\varphi|^2,
\end{equation*} 
so that, combining with \eqref{eqTwistedterm}, we obtain
\begin{equation}\label{eqLapProofIneq2}
      \Delta_{\omega,f} (e^{-\kappa f}|\varphi|^2) =  e^{-\kappa f}|\varphi|^2\left( -\kappa^2 ||df||^2_\omega+2\kappa (2\kappa-1)\lambda - 4\kappa^2 f \right) + e^{-\kappa f}  \Delta_{\omega} |\varphi|^2.
\end{equation}
We next develop the term $\Delta_{\omega} |\varphi|^2$. To this end, recall that any holomorphic function $\varphi$ on $Y$ satisfies (on the open dense subset where $\varphi \neq 0$)
\begin{equation}\label{holomorphic} dd^c|\varphi|^2 = 2id\varphi\wedge d\bar \varphi= \frac{1}{|\varphi|^2}\left(d|\varphi|^2\wedge d^c|\varphi|^2\right).\end{equation}
Furthermore, if $\iota : N\hookrightarrow Y$ denotes the inclusion and $\psi$ is any smooth function on $Y$,  then we have on $N$ \[d^c_{\hat\xi}(\iota^*\psi) = \iota^*(d^c\psi) + (\mathcal{L}_{J\hat{\xi}}\psi)\eta^{\hat \xi}, \qquad dd^c_{\xi}\psi = \iota^*(dd^c\psi) - d \left((\mL_{-J\hat{\xi}} \psi)\eta^{\hat \xi} \right).\]  
We use the previous two identities and \eqref{all-weights} to compute
\[ \Delta_{\omega} |\varphi|^2 = -\tr_{\omega} dd^c_{\hat \xi} |\varphi|^2  = -\tr_{\omega} ( 2id\varphi\wedge d\bar {\varphi}) + 2n\kappa  |\varphi|^2.\] 
Substituting back in \eqref{eqLapProofIneq2} and regrouping the terms yields {
\begin{equation}\label{eqLapProofIneq3}\begin{split}\Delta_{\omega,f} (e^{-\kappa f}|\varphi|^2) = &2\kappa e^{-\kappa f}|\varphi|^2\left( n-\lambda\right)\\&+ 2\kappa^2 e^{-\kappa f}|\varphi|^2\left( -\frac{ ||df||^2_\omega}{2}+2\lambda -2f\right)\\&-e^{-\kappa f}\tr_{\omega}(2id\varphi\wedge d\bar{\varphi}) \end{split} \end{equation} }
{ 
Working outside the zero locus of $\varphi$ for a moment,  we compute
\begin{equation*}\label{eqDerivNorm}||d (e^{-\frac{\kappa f}{2}}|\varphi|)||^2_\omega = \frac{\kappa^2}{4}|\varphi|^2e^{-\kappa f}(||d f ||^2_\omega)  - \kappa|\varphi|e^{-\kappa f} \langle df,d|\varphi|\rangle_{\omega} + e^{-\kappa f}||d|\varphi|||^2_\omega. \end{equation*}} 
{ Using \eqref{eqTwistedterm}, the second term of the right hand side of the above identity may be rewritten and we get}
{ 
\begin{equation}\label{eqDerivNorm2}||d (e^{-\frac{\kappa f}{2}}|\varphi|)||^2_\omega = \frac{\kappa^2}{4}|\varphi|^2e^{-\kappa f}(||d f ||^2_\omega)  - \kappa^2|\varphi|^2 e^{-\kappa f} (\lambda -f) + e^{-\kappa f}||d|\varphi|||^2_\omega. \end{equation}}
{
The second line of \eqref{eqLapProofIneq3} is exactly $- 4 ||d (e^{-\frac{\kappa f}{2}}|\varphi|)||^2_\omega +4e^{-\kappa f}||d|\varphi|||^2_\omega$, so that} 
{
\begin{equation*}\label{eqLapProofIneq4}\begin{split}\Delta_{\omega,f} (e^{-\kappa f}|\varphi|^2) = &2\kappa e^{-\kappa f}|\varphi|^2\left( n-\lambda\right)\\&- 4 ||d (e^{-\frac{\kappa f}{2}}|\varphi|)||^2_\omega +4e^{-\kappa f}||d|\varphi|||^2_\omega\\&-e^{-\kappa f}\tr_{\omega}(2id\varphi\wedge d\bar{\varphi}). \end{split} \end{equation*}}
{Using the second equality in \eqref{holomorphic}  we finally obtain 
\begin{equation}\label{eqLapProofIneq5}\Delta_{\omega,f} (e^{-\kappa f}|\varphi|^2) = 2\kappa e^{-\kappa f}|\varphi|^2\left( n-\lambda\right)- 4 ||d (e^{-\frac{\kappa f}{2}}|\varphi|)||^2_\omega, \end{equation} } 
{which yields the inequality  \eqref{eqCoroWeitzenbockCone} everywhere on $Y$.}

\bigskip
From \eqref{eqCoroWeitzenbockCone}, we  conclude that $(n-\lambda)\ge 0$ by the maximum principle applied on a local quotient of $N$ by $\hat \xi$, near a point of global maximum of $e^{-kf}|\varphi|^2$ on $N$. (Notice that $|\varphi|^2$ cannot be identically zero on $N$ as $\varphi$ is  a non-trivial holomorphic function on $Y$ by assumption.) If $\lambda =n$, then the maximum principle tells us that  $e^{-\kappa f} |\varphi|^2=C$ for some constant $C$ on $N$. By \eqref{all-weights}, we  have on $Y$
$$|\varphi|^2=C r^{2\kappa} e^{\kappa f}.$$ Taking Lie derivative in the direction of $-J\tau_{\hat \xi}$ in both sides of the last equation  (and using \eqref{all-weights} and \eqref{decompose}) gives \[ \lambda = C(|df|^2 +f),\]  which in turn implies that $f$ is a constant and thus $\tau_{\hat \xi}=0$. In that case $\lambda=0$, a contradiction.\end{proof}

\section{A weighted version of the Fujita volume bound: Proof of Theorem~\ref{thm:fujita}}

%%%%%
%%%%

\subsection{The weighted $\beta$-invariant of a polarized manifold and Fujita's volume bound}
Here we recall the setup in \cite{HL, HL2} of valuative characterization  of weighted Ding stability of smooth Fano varieties. 

Let  $\pi: L\ra X$ be a holomorphic line bundle endowed with the action of a real torus $\T \subset \Aut(X, L)$,  covering $\T \subset \Aut(X)$.  Let $W_L \subset \tor^*$ be the convex hull of  normalized weights
\[\left\{ \frac{\alpha}{k} \in \mathfrak{t}^* \;|\;  k\in \N^*, H^0(X,kL)_{\alpha} \neq 0 \right\}\] of the linear action of $\T$ on the vector spaces $H^0(X, kL)$.  When $L$ is ample, similarly to \eqref{moment-map}, one can associate to the chosen lift $\T \subset \Aut(X, L)$ a momentum polytope $P_{L} \subset \tor^*$ and it is well-known that in this case  $P_L=W_L$, see for example \cite[Chapter 8]{Woodward}.

For a smooth weight function $v \in C^\infty(W_L,\R_{>0})$, Han--Li introduced in \cite{HL} the (algebraic) $v$-weighted volume of $L$ as the  limit  
\[ \Vol_v(L) : = \lim_{k\ra +\infty} \frac{n!}{k^{n}} \sum_{\alpha \in \kt^*} v\left(\frac{\alpha}{k}\right)  \dim \left(H^0(X,kL)\right)_{\alpha},\] 
the existence of which is justified in \cite[p.40]{HL}. We set
\[ \Vol(L):= \Vol_1(L).\]
In the case when $L$ is ample, we have (see e.g. \cite[Lemma 14]{lahdili})
\[ \Vol_v(L) = \frac{n!}{(2\pi)^n}  \int_{\Pol_X} v(x) d\mu_{\rm DH},\]
where $d\mu_{\rm DH}$ is introduced through \eqref{DH} for a $\omega \in 2\pi c_1(L)$. It thus follows that in this case \begin{equation}\label{volume-normalization}
    \Vol(L) = c_1(L)^n.\end{equation}

Let now $X$ be a smooth Fano variety and $\T \subset \Aut(X)$ a maximal torus as in the setup of Sect.~\ref{s:notation}. The torus $\T$  admits a canonical lift on $L=-K_X$, and by the above discussion, $W_{-K_X}=\Pol_X$, where $\Pol_X$ is the canonically normalized polytope of $(X, \T)$. Let $\hat X$ be the blow-up of $X$ at a point $p\in X$ fixed by $\T$. We thus have a $\T$-equivariant birational morphism $b : \hat{X} \to X$,  and a $\T$-stable exceptional divisor $D \subset \hat{X}$. Notice that $D$ is, by definition, an instance of a \emph{prime divisor} over $(X, \T)$,  so we can use the theory of weighted $\beta$-invariant~\cite{HL,HL2}, which we recast to our setup below.
 
 Given $x\in \Q$, there exists $m_x \in \N^*$ such that $m_x(-b^*K_X -x[D])$ is a holomorphic line bundle on $\hat X$. Notice that $-b^*m_xK_X -xm_x[D]$ comes with an induced $\T$-action through the blowing-up construction,  and we have a $\T$- equivariant exact sequence of sheaves  \[0\longrightarrow \cO_{\hat X}( -b^*m_xK_X -xm_x[D]) \longrightarrow \cO_{\hat X}( -b^*m_xK_X) \longrightarrow \cO_{D}( -xm_x [D]) \longrightarrow 0,\] showing the inclusion $\left(H^0(X,m_x(-b^*K_X -x[D]))\right)_{\alpha} \subset \left(H^0(X,-m_x K_X)\right)_{\alpha}$ for each weight $\alpha$. It follows that $W_{m_x(-b^*K_X -x[D])} \subset  W_{m_x(-K_X)} =m_x P_X$.  One can therefore define the $v$-weighted volume $\Vol_v(m_x(-b^*K_X -x[D]))$  of $m_x(-b^*K_X-x[D])$ over $\hat X$ as 
\begin{equation*}
\Vol_v(m_x(-b^*K_X -x[D])) : = \lim_{k\ra +\infty} \frac{n!}{ k^{n}} \sum_{\alpha \in \kt^*} v\left(\frac{\alpha}{m_xk}\right)  \dim \left(H^0(\hat{X},k m_x(-b^*K_X -x[D]))\right)_{\alpha}.
\end{equation*}
According to \cite{HL}, this weighted volume is well-defined, continuous in $x$, and homogeneous of order $n$ in $m_x$. We can thus set  
\[ \Vol_v((-b^*K_X -x[D])) : =\frac{1}{m_x^n}\left( \lim_{k\ra +\infty} \frac{n!}{k^{n}} \sum_{\alpha \in \kt^*} v\left(\frac{\alpha}{m_xk}\right)  \dim \left(H^0(\hat{X},k m_x(-b^*K_X -x[D]))\right)_{\alpha}\right),\]
which is independent of $m_x$. The above definition yields the monotonicity property
\begin{equation}\label{vol-monotone} \Vol_{v_1}(-b^*K_X -x[D]) \le \Vol_{v_2}(-b^*K_X -x[D]) \qquad \forall \, v_1 \le v_2 \, \textrm{on} \, \Pol_X, \end{equation}
which we shall use below.

The weighted $\beta$ invariant of $D\subset \hat X$ is then introduced by
\begin{equation}\label{beta-v} \beta_v(D):= A_X(D)\Vol_v(-K_X) - \int_{0}^{\infty}\Vol_v(b^*(-K_X)-x[D]) dx, \end{equation} 
where $A_X(D)$ stands for the log discrepancy of the prime divisor $D$, see \cite[Definition 2.5]{DL23}. In the case when $D \subset \hat X$ is the exceptional divisor as above,  $A_X(D)=n$ (see e.g. \cite{fujita}).

\begin{proof}[Proof of Theorem~\ref{thm:fujita}] From \cite[Theorems 1.7 \& 5.18]{HL} it follows that under the hypothesis of Theorem~\ref{thm:fujita}, we have $\beta_v(D) \geq 0$, i.e.  
\[ \frac{1}{n} A_X(D) \Vol_v(-K_X) =  \Vol_v(-K_X)\ge \frac{1}{n} \int_0^{\infty} \Vol_v(-b^*K_X - x[D]).\] 
Let $m_v:=\underset{\mu\in\Pol_X}{\min}v(\mu)>0$.  By \eqref{vol-monotone} we have 
\[\Vol_{v}(-b^*K_X-x[D]) \geq m_v\Vol(-b^*K_X-x[D])\] whereas by \cite[Theorem 2.3]{fujita}
\[\Vol(b^*K_X-x[D]) \ge \Vol(-K_X)-x^n. \]
The above inequalities  give
\[
\begin{split}
 \Vol_v(-K_X)\geq&\frac{1}{n} \int_0^\infty \Vol_{v}(-b^*K_X-x[D])dx\\
 \geq&\frac{m_v}{n}\int_0^\infty \Vol(-b^*K_X-x[D])dx\\
 \geq& \frac{m_v}{n}\int_0^{(\Vol(-K_X))^{1/n}}\Vol(-b^*K_X-x[D])dx\\
 \geq & \frac{m_v}{n}\int_0^{(\Vol(-K_X))^{1/n}} \left(\Vol(-K_X)-x^{n}\right)dx\\
=&\frac{m_v}{n}\left( (\Vol(-K_X))^{1/n} \Vol(-K_X)-\frac{1}{n+1}(\Vol(-K_X))^{(n+1)/n}\right)\\
=&\frac{m_v}{n+1}(\Vol(-K_X))^{(n+1)/n}
\end{split}
\]
Under the normalization $\Vol_v(-K_X)=\Vol(-K_X)$ for the weight $v$, we obtain
\[
\Vol(-K_X)\leq \left(\frac{n+1}{m_v}\right)^n.
\]
The claim follows by \eqref{volume-normalization}.\end{proof}

\begin{rem} More generally, let $(X,L, \T)$ be a polarized projective manifold and $\T\subset\Aut(X,L)$ a lift of  the action of a real torus in $\Aut(X)$,  with corresponding momentum polytope $\Pol_L\subset\tor^*$. For any two weight functions $v>0$, $w$ on $\Pol_L$,  and a $\T$-stable prime divisor $D$ of $(X, \T)$, one can associate the $(v, w)$-weighted $\beta$-invariant of $D$ by the formula \begin{equation}\label{w-beta} \beta_{v,w}(D):=A_X(D)\Vol_v(L)+ \frac{1}{2} \int_0^\infty \Vol_w(L-x[D])dx+ \int_0^\infty \Vol'_v(L-x[D])\cdot [K_X] dx, \end{equation} where, as in \cite{DL23},  $\Vol'_v(L-x[D])\cdot [K_X] dx$ stands for the (formal) derivative  of $\Vol_v$ in the direction of the canonical bundle $K_X$. The above formula extends  the ``unweighted'' situation with $v=1, w= 2n \left(\frac{c_1(L)\cdot L^{n-1}}{L^n}\right)$  studied in \cite{DL23}. In general, in \eqref{w-beta},  we assume the differentiability of the extension of $\Vol_v$ on $\R$-line bundles in each direction, a fact established in the case $v=1$ in \cite{BFJ} but  which is open in general.  

Similarly to \cite[Corollary 3.11]{DL23}, using integration by parts,  one can rewrite \eqref{w-beta} as 
\[\beta_{v,w}(D)=A_X(D)\Vol_v(L)+ \frac{1}{2}\int_0^\infty \Vol_{w-(\tilde{v}+v)}(L-x[D])dx+ \int_0^\infty \Vol'_v(L-x[D])\cdot [L+K_X] dx. \] 
 When $X$ is Fano, $L=-K_X$ and   $w=\tilde v := 2(n + d\log v)v$,
the above reduces to \eqref{beta-v}, showing the consistency of \eqref{w-beta}. It would be interesting to show that, similarly to the unweighted case~\cite{DL23},  the value $\beta_{v,w}(D)$ is related to the $(v,w)$-weighted Donaldson-Futaki invariant~\cite{lahdili} of a certain test configuration.
\end{rem}

\appendix
\section{Smooth deformations of weighted cscK metrics}\label{a:LeBrun-Simanca}

Let $X$ be a compact K\"ahler manifold endowed with a K\"ahler class $\alpha \in H^{1,1}(X,\mathbb{R})$ and $\T\subset {\rm Aut}_{r}(X)$ a maximal torus with momentum polytope $\Pol \subset \tor^*$. Let $v, w\in C^{\infty}(\Pol), \, v >0$ be a pair of weight functions on $\Pol$. Recall that  $(v, w)$-weighted Futaki invariant (see \cite{lahdili}) vanishes if for some (and hence any) $\T$-invariant K\"ahler metric $\omega\in \alpha$  and  for any affine-linear function $\ell(x)$ on $\tor^*$,  
\[\Fut_{v,w} (\ell) : = \int_{X} \left(\Scal_{v}(\omega) - w(\mu_{\omega_\omega})\right) \, \ell(\mu_{\omega}) \omega^{[n]} =0,\]
where $\Scal_v(\omega)$ stands for the $v$-scalar curvature of $\omega$ given by
\[\Scal_v(\omega):= v(\mu_{\omega}) \Scal(\omega) + 2\Delta_{\omega} v(\mu_{\omega})  + \big\langle g_{\omega}, \mu_{\omega}^*\left(\Hess(v)\right)\big\rangle, \]
with  $\Scal(\omega)$ being the usual scalar curvature of the Riemannian metric $g_\omega$  associated to $\omega$,  $\Delta_{\omega}$ standing for the Laplace operator of $g_{\omega}$,  and  $\langle \cdot, \cdot \rangle$  denoting the contraction between the smooth $\tor^*\otimes \tor^*$-valued function $g_{\omega}$ on $X$  (the restriction of the Riemannian metric $g_{\omega}$  to $\tor \subset C^{\infty}(X, TX)$)  and the smooth $\tor\otimes \tor$-valued function ${\mu_{\omega}}^*\left(\Hess(v)\right)$ on $X$ (given by the pull-back by $\mu_{\omega}$ of ${\rm Hess}(v) \in C^{\infty}(\Pol, \tor\otimes \tor)$). 

{ The following result is a slight modification of  \cite[Theorem~B2]{lahdili} (see also \cite[Theorem~6.1.2]{Hallam-thesis} for the special case when $w$ is fixed),  where the weight functions $(v_t, w_t)$ are supposed to be positive on $\Pol$.}

\begin{lemma}\label{l:(v,w)-LeBrun-Simanca}  Let $\{(v_t, w_t), \, \, v_t, w_t \in C^{\infty}(\Pol), \, v_t>0,  \, \,  t\in U \subset \R^k\}$ be a finite dimensional smooth family of weight  functions on $\Pol$, parametrized by a neighbourhood $U$ of $0$ and satisfying  $\Fut_{v_t, w_t}\equiv 0$ for all $t\in U$. Suppose $\omega_0$ is a $\T$-invariant K\"ahler metric in $\alpha$ such that
\[ \Scal_{v_0}(\omega_0)= w_0(\mu_{\omega_0}).\] 
Then, there exists $\epsilon >0$ and a differentiable family of smooth $\T$-invariant K\"ahler metrics $\omega_t \in \alpha$, such that for $||t||<\epsilon$
\[
\Scal_{v_t}(\omega_t)=w_t(\mu_{\omega_t}),
\]
where $\mu_{\omega_t}: X\to \Pol$ is the $\omega_t$-momentum map of $\T$ normalized by $\mu_{\omega_t}(X)=\Pol$.
\end{lemma}
\begin{proof}  The proof is a standard adaptation of the arguments in \cite{LS}, using the computations in \cite{lahdili}.

Consider the Fr\'echet space $C^{\infty}(X)^{\T}$ of smooth $\T$-invariant functions on $X$ and denote by  $\mathcal{K}^{\T}_{\omega_0}(X)$ the subset of $\T$-invariant $\omega_0$-relative K\"ahler potentials; for any $\varphi\in \mathcal{K}^{\T}_{\omega_0}(X)$, we let $\omega_\varphi:=\omega_0+dd^{c}\varphi>0$ be the corresponding K\"ahler structure on $X$, $\mu_{\varphi}:= \mu_{\omega_0} + d^c\varphi$ be the normalized $\omega_{\varphi}$ momentum map, and we let  \[ \mathcal{P}_{\omega_{\varphi}}:=\{ \ell(\mu_{\varphi}), \, \ell \in  \Aff(\Pol)\}\]
denote the space of all $\omega_{\varphi}$-Hamiltonians of elements of ${\rm Lie}(\T)$. We denote by 
\[\Pi_{\omega_{\varphi}} : C^{\infty}(X)^{\T} \to \mathcal{P}_{\omega_{\varphi}}\]  the orthogonal projection with respect to the  global $L^2$ product defined by $\omega_{\varphi}^{[n]}$ and let
\[C^{\infty}_{\perp}(X)^{\T}:=(1-\Pi_{\omega_0})(C^{\infty}(X)^{\T})\]
be the space of smooth $\T$-invariant functions on $X$, which are $L^{2}$-orthogonal to the finite dimensional space $\mathcal{P}_{\omega_0}$. 

We consider the map $\Psi:U\times \left(\mathcal{K}^{\T}_{\omega_0}(X)\cap C^{\infty}_{\perp}(X)^{\T}\right) \to\R^k\times C^{\infty}_{\perp}(X)^{\T}$ defined by
\[
\Psi(t,\varphi):=\Big(t,(1-\Pi_{\omega_0})\left(\Scal_{v_t}(\omega_\varphi)-w_t(\mu_\varphi)\right)\Big).
\]

By the computations in \cite{lahdili} and using that $\Scal_{v_0}(\omega_0)= w_0(\mu_{\omega_0})$, the differential  $(T_{(0,0)} \Psi)$ of $\Psi$ at $(0,0)$  in the direction of $\varphi$ is given by
\[
(T_{(0,0)}\Psi)(0,\dot \varphi)=\left(0, -2\left((1-\Pi_{\omega_0}\right)\left(\delta_{\omega_0}\delta_{\omega_0}(v_0(\mu_{\omega_0})(Dd\dot \varphi)^{-} \right)\right),
\]
where $\delta_{\omega_0} T := -\sum_{i=1}^{2n}(D_{e_i}T)(e_i, \cdot)$
denotes the codifferential of a $(p,0)$-tensor $T$ with respect to $\omega_0$, $D$ stands for the Levi-Civita connection of $\omega_0$ and $(D d \dot \varphi)^{-}$ is the $(2,0)+(0,2)$ part of the Hessian of $\dot \varphi \in C^{\infty}(X)^{\T}$.

Notice that the weighted Lichnerowicz operator ${\mathbb L}_{\omega_0}(\dot \varphi):=\delta_{\omega_0}\delta_{\omega_0}\Big(v_0(\mu_{\omega_0})(Dd\dot\varphi)^{-} \Big)$ is a linear self-adjoint operator with respect to the  $L^2$ inner product of $\omega_0$,  whose kernel in $C^{\infty}(X)^{\T}$  consists of all $\T$-invariant Killing potentials of $\omega_0$.  As $\T \subset \Aut_r(X)$ is a maximal torus, we conclude that the kernel of ${\mathbb L}_{\omega_0}$ is $\mathcal{P}_{\omega_0}$ and its image is  in $C^{\infty}_{\perp}(X)^{\T}$. Thus, by standard linear elliptic theory, { ${\mathbb L}_{\omega_0}$ (and hence $\dot \varphi \to (T_{0,0} \Psi) (0, \dot \varphi)$) is an isomorphism when restricted to $C^{\infty}_{\perp}(X)^{\T}$, and hence also when viewed as a linear map between the Banach spaces $W^{k+4}_{\perp}(X)^{\T}$ and $W^{k}_{\perp}(X)^{\T}$, where  for any positive integer $k$,  $W^k(X)^{\T}$ denotes the completion in $L^2(X,\omega_0)$ of $C^{\infty}(X)^{\T}$ with respect to the Sobolev norm involving derivatives up to order $k$,  and $W^k_{\perp}(X)^{\T}:=(1-\Pi_{\omega_0})(W^k(X)^{\T})$. By Sobolev's embedding theorem, for $k$ large enough, we have  $W^{k+4}(X)^{\T} \subset C^4(X)^{\T}$.  As $\Scal_{v_t}(\omega_\varphi)$ is a 4th order quasilinear operator in $\varphi$, $\Psi(t, \varphi)$ admits an extension (still denoted by $\Psi$) to a $C^1$ map $\Psi: U\times \left(\mathcal{K}^{\T}_{\omega_0}(X) \cap W^{k+4}_{\perp}(X)^{\T}\right) \to \R^k \times W^{k}_{\perp}(X)^{\T}$, where  the open subspace $\mathcal{K}^{\T}_{\omega_0}(X) \cap W^{k+4}_{\perp}(X)^{\T}$ is of $C^2$-K\"ahler metrics in $\alpha$.  Applying the implicit function theorem  to $\Psi$ and using bootstraping elliptic regularity for the map $\varphi \to \Scal_v(\omega_{\varphi})$}, we conclude that we can find a differentaible path of potentials $\varphi_t \in \mathcal{K}^{\T}_{\omega_0}(X)\cap C^{\infty}_{\perp}(X)^{\T}, \, |t|<\epsilon_0$ with $\varphi_0\equiv 0$, such that
\[\Scal_{v_t}(\omega_{\varphi_t})-w_t(\mu_{\varphi_t}) = \Pi_{\omega_0}\left(\Scal_{v_t}(\omega_{\varphi_t})-w_t(\mu_{\varphi_t})\right).\]
Now, by the assumption $\Fut_{v_t, w_t} =0$,  $\Pi_{\omega_{\varphi_t}}\left(\Scal_{v_t}(\omega_{\varphi_t})-w_t(\mu_{\varphi_t})\right) =0$. Thus, for $||t||<\epsilon$ small enough, one also has $\Pi_{\omega_{0}}(\Scal_{v_t}(\omega_{\varphi_t})-w_t(\mu_{\varphi_t}) =0$, which finishes the proof.
\end{proof}
The above deformation result applies in particular in the case of a $v_0$-soliton and a smooth family  weights $v_t \in \mathcal{F}(X)$.
\begin{cor} \label{c:Lebrun-Simanca} Let $(X, \T)$ be a smooth Fano manifold with canonically normalized polytope $\Pol_X$ and $\{v_t\in C^{\infty}(\Pol_X), \, t\in U \subset \R^k\} $ be a finite dimensional smooth family of positive weight functions,  parameterized by a neighbourhood $0\in U\subset  \R^k$,   which satisfy $\Fut_{v_t}\equiv 0$ for all $t\in U$. Suppose $v_0\in \mathcal{S}(X)$, i.e. $X$ admits a $\T$-invariant $v_0$-soliton. Then there exits $\varepsilon >0$ such that for all $||t||< \varepsilon$, $v_t \in \mathcal{S}(X)$.
\end{cor} 
\begin{proof} By \cite[Prop.1]{AJL}, a $v$-soliton  $\omega$ on $(X, \T, 2\pi c_1(X))$ can be characterized by the property 
\[ \Scal_v(\omega)= \tilde v (\mu_{\omega}),\]
where $\tilde v(x):= 2(n+ \langle d\log v, x\rangle)v(x)$. Furthermore, by \cite[Lemma B1]{AJL} we have
\[\Fut_v \equiv 0 \, \Leftrightarrow \, \Fut_{v, \tilde v}\equiv 0.\]
The claim then follows from Lemma~\ref{l:(v,w)-LeBrun-Simanca}.
\end{proof}

\section{Curvature identities for $v$-solitons}\label{a:Curvature-Identities}

\begin{lemma}\cite{PW}\label{l:identity-scal} Suppose $(g, J, \omega)$ is a $\T$-invariant  K\"ahler metric which is a KRS, i.e. satisfies
\begin{equation}\label{gradient-soliton} \Ric(\omega) - \omega = \frac{1}{2} dd^c f,\end{equation}
where $f$ is  the potential of a real Killing and holomorphic vector field $\tau = J grad_g f$.
Then the scalar curvature $\Scal(\omega)$  satisfies the following identity 
\begin{equation}\label{scal-identity}\frac{1}{2} \Delta \Scal(\omega) - \langle d\Scal(\omega), df\rangle_g + \Scal(\omega)=  ||\Ric_g||^2. \end{equation}
\end{lemma}
\begin{proof}  
This is the forth formula in \cite[Lemma 2.5]{PW} (which generally holds for any gradient Ricci soliton). We reproduce its proof in our K\"ahler situation for convenience of the Reader.

We denote by $\Ric_g$ the (symmetric) Ricci tensor of $g$,  so that the (K\"ahler) Ricci form is $\Ric(\omega)(\cdot, \cdot)= \langle \Ric_gJ \cdot, \cdot \rangle$. 
As $\tau := J\grad_g f$ is a Killing vector field, we have 
\[ \nabla d^c f = \frac{1}{2} dd^c f,\]
and, by Bochner's identity, 
$ \Delta (d^cf) =  2\Ric_g(d^cf)$ or, equivalently,
\begin{equation}\label{killing} 
-d\left(\Delta f\right) = 2\Ric(\omega)(\tau).
\end{equation}
Taking trace with respect to $\omega$ in the soliton equation \eqref{gradient-soliton} yields
\[ \Scal(\omega) -2n = - \Delta f,\]
so that we have 
\begin{equation}\label{Scal} d\Scal(\omega) =- \Delta d f= 2\Ric(\omega)(\tau), \qquad \Delta \Scal(\omega)= - \Delta \Delta f. \end{equation}
Applying the covariant derivative $\nabla_X$ in \eqref{gradient-soliton} yields
\begin{equation}\label{1-jet} \nabla_X \left(\Ric(\omega)\right) = \frac{1}{2}\nabla_X (dd^c f).\end{equation}
Taking interior product with $\xi$  in \eqref{1-jet} and using \eqref{Scal} and \eqref{gradient-soliton} again gives
\[
\begin{split}
   \frac{1}{2} (\nabla_X dd^c f)(\tau) &= \left(\nabla_X \Ric(\omega)\right)(\tau) = \nabla_X(\Ric(\omega)(\tau)) - \Ric(\omega)(\nabla_X \tau)  \\
    &= \frac{1}{2} \nabla_X d\Scal(\omega) -\Ric(\omega)(\Ric_g(JX) - JX) \\
    &=  \frac{1}{2}\nabla_X d\Scal(\omega) + \Ric_g(\Ric_g(X) - X).
    \end{split} \]
Taking trace gives 
\[ \frac{1}{2} (\Delta d^cf )(\tau)=-\frac{1}{2} \Delta \Scal(\omega) + ||\Ric_g||^2 - \Scal(\omega).\]
By \eqref{Scal}  we  have $\frac{1}{2} (\Delta d^cf )(\tau)= -\frac{1}{2}\langle d\Scal, df \rangle$, so we  finally get \eqref{scal-identity}. \end{proof}

As an immediate corollary, we get the following fact, which we presume is well-known to experts,  but which we were not able to trace in the literature.
\begin{cor}\label{c:scal>0} Suppose $(N, \Ds, J, \hat \xi, \eta_{\hat \xi})$ is a compact Sasaki manifold which is a transversal KRS. Then the transversal scalar curvature is everywhere positive.
\end{cor}
\begin{proof} We denote by $\Scal(\omega)$ the transversal scalar curvature and consider a local $\hat\xi$-quotient $(U, J, \omega, g)$ around a point of global minimum of $\Scal(\omega)$ on $N$. Notice that $(U, J)$ is a (open) complex manifold and $\omega$ descends to define a KRS on $U$. By Lemma~\ref{l:identity-scal}, the smooth function $\Scal(\omega)$ satisfies on $U$ the identity \eqref{scal-identity}.
As $\Scal(\omega)$ achieves its minimum in the interior of $U$, we get by the strong maximum principle $\Scal(\omega) >0$ on $U$, and hence on $N$.
\end{proof}

\end{document}